\documentclass[a4paper,11pt]{amsart}
\usepackage{fullpage}

\usepackage[foot]{amsaddr}

\usepackage{hyperref}
\hypersetup{
    colorlinks,
    citecolor=green,
    filecolor=black,
    linkcolor=blue,
    urlcolor=blue
}
\hypersetup{linktocpage}

\usepackage{cite}
\usepackage{graphicx}

\usepackage{amsbsy}
\usepackage{latexsym}
\usepackage{amsfonts}
\usepackage{amssymb}
\usepackage[usenames]{color}
\usepackage{amsmath,amsthm}
\usepackage{enumerate}
\usepackage{tikz}
\usetikzlibrary{arrows}
\usepackage{mathdots}
\usepackage{mathtools}

\usepackage{stmaryrd}
\usepackage{dsfont}
\usepackage{color}

\usepackage{todonotes}

\newcommand{\rev}[1]{#1}

\providecommand{\abs}[1]{\lvert#1\rvert}
\providecommand{\bigabs}[1]{\bigl\lvert#1\bigr\rvert}

\providecommand{\biggabs}[1]{\biggl\lvert#1\biggr\rvert}

\providecommand{\norm}[1]{\lVert#1\rVert}

\providecommand{\floor}[1]{\lfloor#1\rfloor}
\providecommand{\ceil}[1]{\lceil#1\rceil}

\newtheorem{theorem}{Theorem}
\newtheorem{lemma}[theorem]{Lemma}

\newtheorem{proposition}[theorem]{Proposition}

\theoremstyle{definition}

\theoremstyle{remark}

\newcommand{\MAT}{Mat\'{e}rn }

\newcommand{\cN}{{\mathcal{N}}}

\newcommand{\Chi}{\raise .3ex
\hbox{\large $\chi$}}

\newcommand{\R}{\mathbb{R}}

\newcommand{\N}{\mathbb{N}}
\newcommand{\Z}{\mathbb{Z}}
\newcommand{\C}{\mathbb{C}}

\newcommand\supp{\mathop{\rm supp}}

\usepackage{todonotes}

\newcommand{\ext}{{\mathrm{ext}}}

\newcommand\dx{\mathrm{d}}

\newcommand{\Sigmaext}{\Sigma^{\mathrm{ext}}}
\newcommand{\Qext}{Q^{\mathrm{ext}}}

\newcommand{\rd}{\mathrm d} 
\newcommand{\ri}{\mathrm i} 
\usepackage{cleveref}
\newcommand{\bsx}{x}
\newcommand{\bsY}{Y}
\newcommand{\bsZ}{V}

\newcommand{\bk}{k}
\newcommand{\bsk}{k}
\newcommand{\len}{{\gamma}}
\newcommand{\bbZ}{\mathbb{Z}}
\newcommand{\Rext}{\Sigmaext}
\newcommand{\bbR}{\mathbb{R}}
\newcommand{\rhoext}{\rho^{\mathrm{ext}}}
\newcommand{\Lambdaext}{\Lambda^{\mathrm{ext}}}
\newcommand{\bsV}{V}

\newcommand{\oZ}{\overline{\mathbb{Z}}}
\newcommand{\bn}{n}
\newcommand{\bx}{x}
\newcommand{\bm}{m}
\newcommand{\cR}{\mathcal{R}}
\newcommand{\cRext}{{\cR}^{\rm ext}}
\newcommand{\bbZbar}{\overline{\mathbb{Z}}}
\newcommand{\bsomega}{\omega}
\newcommand{\bsxi}{\xi}
\newcommand{\bsalpha}{\alpha}
\newcommand{\bsbeta}{\beta}

\numberwithin{equation}{section}

\title{Unified Analysis of Periodization-Based Sampling Methods for Mat\'ern Covariances}
\author{Markus Bachmayr$^1$}
\address{\rm $^1$ Institut f\"ur Mathematik, Johannes Gutenberg-Universit\"at Mainz, Staudingerweg 9, 55128 Mainz, Germany}
\email[Markus Bachmayr]{bachmayr@uni-mainz.de}
\author{Ivan G.  Graham$^2$}
\address{\rm $^2$ Department of Mathematical Sciences, University of Bath, Bath BA2 7AY, United Kingdom}
\email{i.g.graham@bath.ac.uk}
\author{Van Kien Nguyen$^3$}
\address{\rm $^3$ Department of Mathematics, University of Transport and Communications, No.3 Cau Giay Street, Lang
      Thuong Ward, Dong Da District, Hanoi, Vietnam}
\email{kiennv@utc.edu.vn}
\author{Robert Scheichl$^4$}
\address{\rm $^4$ Institute for Applied Mathematics and Interdisciplinary Center for Scientific Computing (IWR), Ruprecht-Karls-Universit\"at Heidelberg, Im Neuenheimer Feld 205, 69120 Heidelberg, Germany}
\email{r.scheichl@uni-heidelberg.de}
\thanks{M.B.\ and V.K.N.\ acknowledge support by the Hausdorff Center of Mathematics, University of Bonn. V.K.N.\ would like to thank Vietnam Institute for Advanced Study in Mathematics for hospitality during his visit in summer 2018 when part of the work on this paper was undertaken. All authors thank the Johann Radon Institute for Computational and Applied Mathematics, Linz, for support to attend the Workshop on ``Frontier Technologies for High-Dimensional Problems and Uncertainty Quantification'', December 2018, during which part of this work was carried out. I.G.G.\ and R.S.\ thank Dirk Nuyens (Leuven) for useful discussions}
\date{\today}

\begin{document}

\begin{abstract}
The periodization of a stationary Gaussian random field on a sufficiently large torus comprising the spatial domain of interest is the basis of various efficient computational methods, such as the classical circulant embedding technique using the fast Fourier transform for generating samples on uniform grids.
For the family of Mat\'ern covariances with smoothness index  $\nu$ and correlation length  $\lambda$, we analyse the
  nonsmooth periodization (corresponding to classical circulant embedding)  and an alternative procedure using   a smooth  truncation of the covariance function. We solve two open problems: the first concerning the $\nu$-dependent asymptotic decay of eigenvalues of the resulting circulant in the nonsmooth case, the second concerning the required size in terms of $\nu$, $\lambda$ of the torus when
using a smooth periodization. In doing this we arrive at a complete characterisation of the performance of
these two approaches. Both our theoretical estimates and the  numerical tests provided here show substantial advantages of smooth truncation. 
\smallskip

\noindent \emph{Keywords.}  stationary Gaussian random fields, circulant embedding, periodization, Mat\'ern covariances
\smallskip

\noindent \emph{Mathematics Subject Classification.} {60G15, 60G60, 42B05, 65T40}
\end{abstract}

\maketitle

\section{Introduction}

The simulation of Gaussian random fields with specified  covariance is a fundamental  task in computational statistics with a wide
  range of applications -- see, for example, \cite{DN97,WC94,GSPSJ06}. Since these  fields  provide natural tools  for spatial
  modeling  under 
uncertainty, they  have
played a   fundamental r\^{o}le in  the modern  field of uncertainty quantification (UQ).
Thus the construction and analysis of  efficient methods for sampling such fields is of widespread  interest.

In this paper we consider  a  class of algorithms for sampling stationary fields 
  based on (artificial)  periodization of the field,  and  fast Fourier transform. These algorithms 
  enjoy log-linear complexity  in  the number   of \rev{spatial points sampled}. While \rev{these algorithms are well-known}  in computational statistics e.g., \cite{DN97,WC94} and are widely applied in UQ  \cite{GKNSS10,LPS,GKNSS,BCM, BCDM},
  so far they have been subjected to  relatively little rigorous numerical analysis.   This paper provides a
  detailed novel analysis of these algorithms  in the case of the important family of
  \MAT covariances. \rev{In particular we present new results} related to their \rev{efficiency} and \rev{to the preservation of spectral properties in the periodization}.  

\rev{To give more context, we shall be concerned with}  Gaussian random fields 
 $Z= Z(\bsx, \omega)$ which are assumed \rev{to be} stationary, \rev{that is},
 \begin{align}
   \label{eq:covar}\mathbb{E}[(Z(\bsx, \cdot) - \overline{Z}(\bsx)) (Z(\bsx',\cdot) - \overline{Z}(\bsx'))] =: \rho(\bsx - \bsx')
 , \quad \bsx, \bsx' \in D, \end{align} 
  where  $\overline{Z}(\bsx) = \mathbb{E}[Z(\bsx, \cdot)]$ and 
 where  $ D$ is a bounded spatial domain in $d$ dimensions. Given $n$ sampling points
 $\{ \bsx_j\}_{j=1}^n$, our task is  to sample    the multivariate Gaussian
 $\bsZ : = \{ Z(\bsx_i)\}_{i=1}^n \in  \mathcal{N}(0, \Sigma)$, where
 \begin{align} \label{covmatrix} \Sigma_{i,j}  = \rho(\bsx_i -  \bsx_j) , \quad i,j = 1, \ldots, n. \end{align}
 Since $\Sigma$ is symmetric positive semidefinite, this can in principle be done by  performing a factorisation \begin{align} \label{fact} \Sigma = F F^\top, \end{align}
 with $F = \Sigma^{1/2}$ being one possible choice,  from which the desired samples are provided by the product $F\bsY$   where  $ \bsY \in \mathcal{N}(0,I)$.
 When $n$ is large, a naive direct factorisation is prohibitively expensive.
 A variety of methods for drawing samples have been developed that either perform approximate sampling, for
 instance by using a less costly inexact factorisation (e.g., \cite{HPS,FKS}), or maintain exactness of the distribution by using an exact factorisation that exploits additional structure in $\Sigma$. The methods considered in this work follow the latter approach by embedding into a problem with \emph{periodic} structure.

\rev{The methods we consider}
are based on the observation that in many applications the sampling points can
  be placed uniformly on a rectangular grid, as justified theoretically, \rev{for instance}, in \cite{GKNSS18}. Then the $n \times n$ symmetric
  positive semidefinite matrix $\Sigma$ is block Toeplitz with
  Toeplitz blocks \rev{due to \eqref{covmatrix}}. Via  periodization, 
  $\Sigma$ \rev{is}  embedded   into an $m\times m$  block  circulant matrix $\Sigmaext$ with circulant blocks
  (\rev{with  $m > n$, but $m$ typically proportional to $n$}),
  \rev{which is}  diagonalized using  FFT (with log-linear complexity)
\rev{to provide} the spectral decomposition 
\begin{equation}
\label{spectral_decomp}
\rev{\Sigmaext = \Qext \Lambdaext (\Qext)^\top,}
\end{equation}
with  $\Lambdaext$ diagonal and  containing the eigenvalues of $\Sigmaext$ \rev{and $\Qext$ being a Fourier matrix}. Provided these eigenvalues are non-negative, the required $F$ in \eqref{fact} can be computed by taking  $n$ appropriate rows of the square root
  of $\Sigmaext$.  Then samples of the grid values of $Z$  can be drawn by 
  \rev{first drawing a random vector $(y_j)_{j=1,\ldots,m}$ with $y_j\sim \cN(0,1)$ i.i.d.,}
  then computing
  \begin{align}\label{samples}
    \bsZ^\text{ext}  =  \sum_{j=1}^m y_j \sqrt{\Lambdaext_j} \, q_j
  \end{align}
  \rev{using the FFT,} with   $q_j$   the columns of   $\Qext$ 
  and $\Lambdaext_j$  the  corresponding eigenvalues of
  $\Sigmaext$. 
  Note that \eqref{samples} is the Karhunen--Lo\`{e}ve (KL) expansion of the random vector $\bsZ^\text{ext}$.

\rev{Finally, a sample of $\bsZ$ is obtained by extracting from
  $\bsZ^\text{ext}$ the entries corresponding to the original grid
  points. Proceeding in this manner yields exact samples of $Z$ at the
  grid points, provided that $\Sigmaext$ is positive
  semidefinite. \rev{Positive definiteness} can be verified \emph{a
    posteriori} by checking non-negativity of the computed entries of
  $\Lambdaext$ \rev{and} is guaranteed to be satisfied for
  sufficiently large $m$ under fairly general conditions.} \rev{
  For instance, it was shown in \cite[Thm.~2.3]{GKNSS} that this holds true under weak integrability and regularity assumptions on $\rho$.} 

\rev{It thus suffices to iteratively increase $m$ until $\Sigmaext$ is positive semidefinite to obtain a reliable sampling method.
This paper is concerned with the following two key questions
concerning this procedure.}

 \begin{itemize}
   \item[I] \rev{ How large does the extended size $m$ need to be, compared to the cardinality $n$ of the original grid?
    This completely determines the \emph{efficiency} of the sampling scheme.
    The required $m$ depends both on the covariance function $\rho$ and on the type of periodization.}\smallskip
   \item[II] \rev{Do the eigenvalues} $\{\Lambdaext_j: j = 1, \ldots,
     m\}$ maintain a \rev{consistent} rate of decay as $n$ (and hence $m$)
   increases? \rev{This determines the efficiency of numerical methods that build on the decomposition \eqref{samples}, since faster decay of the eigenvalues reduces the number of independent variables which are effectively needed to describe $Z$. 
}
  \end{itemize} 
The answer to  Question II is particularly relevant in several areas
of UQ. For example,  in the analysis of  Quasi-Monte Carlo integration
methods, the rate of decay of the terms in the discrete KL expansion
\eqref{samples} plays a key r\^{o}le in the \rev{convergence theory of QMC
for high-dimensional problems \cite{GKNSS18}}. %
\rev{The rate of decay of the $\Lambdaext_j$ (as $m$ increases) is intricately linked to
  the rate of decay of the exact KL
     eigenvalues $\lambda_j$  of the
     covariance operator with kernel $\rho$ defined on $L^2(D)$, which appear in the KL expansion of the continuous field $Z$,
     \begin{equation}\label{standardkl}
   Z = \sum_{j = 1}^\infty y_j \sqrt{\lambda_j} \varphi_j, \quad y_j \sim \cN(0,1)\;\text{ i.i.d.,}
 \end{equation}
 where $\varphi_j$ are the $L^2(D)$-orthonormal eigenfunctions of the covariance operator.
 Thus, this is a question of how well the properties of the spectrum of the original covariance operator are preserved by the periodization.
}  

\rev{While partial answers to Questions I and II have appeared
  elsewhere \cite{BCM,GKNSS}, this paper gives a full answer to both
  questions in the case where the covariance is of \MAT type,  and in
  the context of two different methods of periodization, both in use
  in practice.} 
\rev{Our results provide a complete quantitative characterization in terms of the 
 parameters of the \MAT covariances. Although we focus on this class of covariances to avoid further complication of the already substantial technical difficulties, the techniques based on suitable cutoff functions that we use in our main results may be more generally applicable. In our concrete arguments, the combination of the exponential decay of $\rho$ and the algebraic decay of its Fourier transform $\hat\rho$ towards infinity, which holds in particular for \MAT covariances, plays an important r\^ole.}

Before describing our main results, we briefly describe the periodization in the case
  of physical dimension $d=1$; higher dimensions are entirely analogous.
Since the sampling domain is bounded, without loss of generality we   assume  that it is contained in $ [-1/2,1/2]$,  and so the covariance $\rho$ is only evaluated  on the domain  $[-1,1]$. For any
  $\gamma \geq  1$ we construct a $2\gamma$-periodic extension of $\rho$ as follows: First choose a cut-off function
  $\varphi$ with the property that
  \begin{align} \label{cutoff}
    \varphi = 1 \quad \text{on} \quad [-1,1],  \quad \text{and} \quad \varphi = 0 \quad \text{on} \quad \mathbb{R} \backslash [-\kappa, \kappa], \quad \text{where} \quad \kappa:= 2 \gamma -1 \geq \gamma.  \end{align}
Then we define $\rhoext$ on $\mathbb{R}$ as the infinite sum of shifts of $\rho \varphi$: 
\begin{align}\label{period}
   \rho^\ext(x) = \sum_{n \in \Z} \bigl( \rho \varphi \bigr)(x + 2\gamma n),  \quad x \in \R\,.
 \end{align}
 Clearly, $\rhoext$ is $2 \gamma$-periodic. Moreover, when  $x \in [-1,1]$ and $0 \not = n \in \mathbb{Z}$, 
 we have $x + 2 n \gamma \in \mathbb{R} \backslash (-\kappa , \kappa)$,  and  \eqref{period} collapses to a single  term, yielding 
 $\rhoext = \rho $ on $[-1,1]$.  We shall discuss in detail two examples corresponding to different choices of $\varphi$:
 \begin{align}
   \label{ex1}   \text{\emph{classical  periodization:}}&  \quad  \varphi = \Chi_{[-\gamma, \gamma)} \,,\  \text{the characteristic function of $[-\gamma, \gamma)$;} \\
   \label{ex2}   \text{\emph{smooth  periodization:}}&  \quad  \gamma > 1 \ \text{and} \  \varphi \in C^\infty_0(\mathbb{R}),    \ \text{with  }  \
                                    \supp(\varphi) = [-\kappa, \kappa] .
 \end{align}
 In \eqref{ex1}, the
 periodization is obtained by simply repeating the function $\rho|_{[-\gamma, \gamma]}$
 periodically. It is \rev{easy} to implement but has the
 disadvantage that artificial non-smoothness is introduced at the points  $\{2n \gamma: 0 \not = n \in \mathbb{Z}\} $.   By contrast, in the smooth periodization \eqref{ex2},  the function $\rho \varphi$ has the same smoothness properties on $\mathbb{R}$ as $\rho$ but    is  supported on $[-\kappa, \kappa]$.   \rev{An illustration of \eqref{period} for the choices \eqref{ex1} and \eqref{ex2} is given in Figure \ref{fig:periodization}.}
\rev{\begin{figure}
\begin{tabular}{cc}
(a) & (b) \\[-1pt]
\includegraphics[height=4cm]{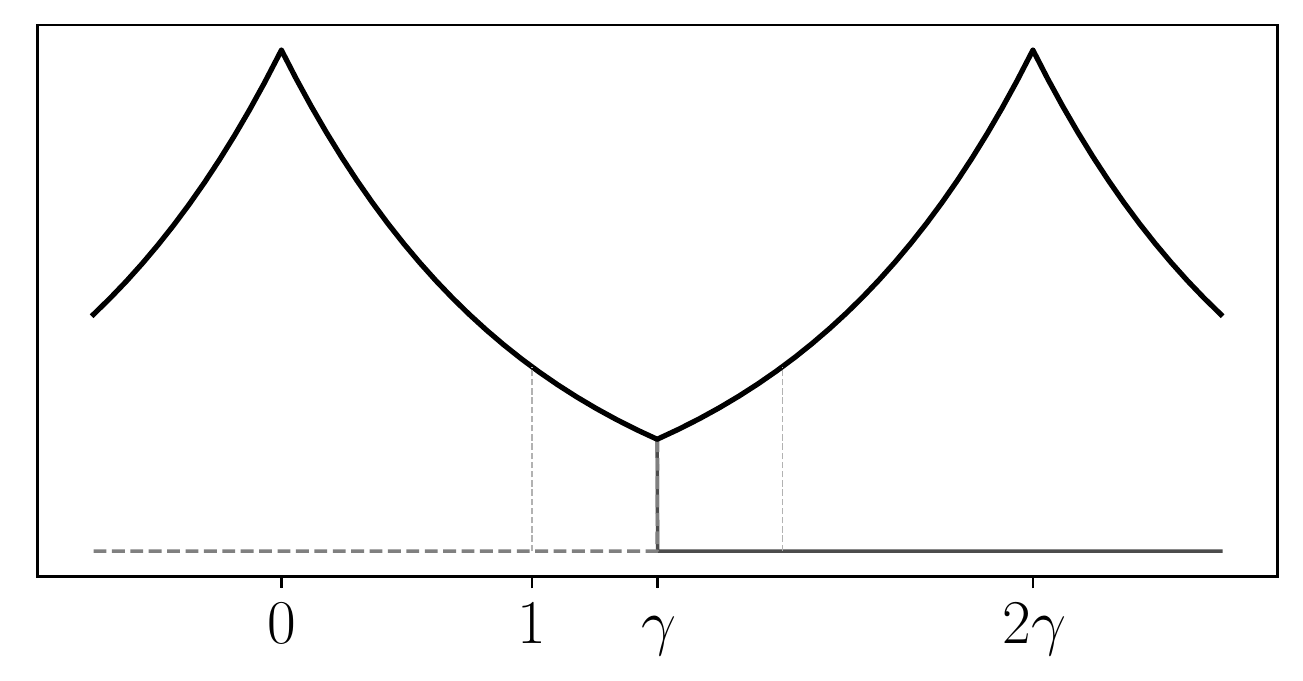} & \includegraphics[height=4cm]{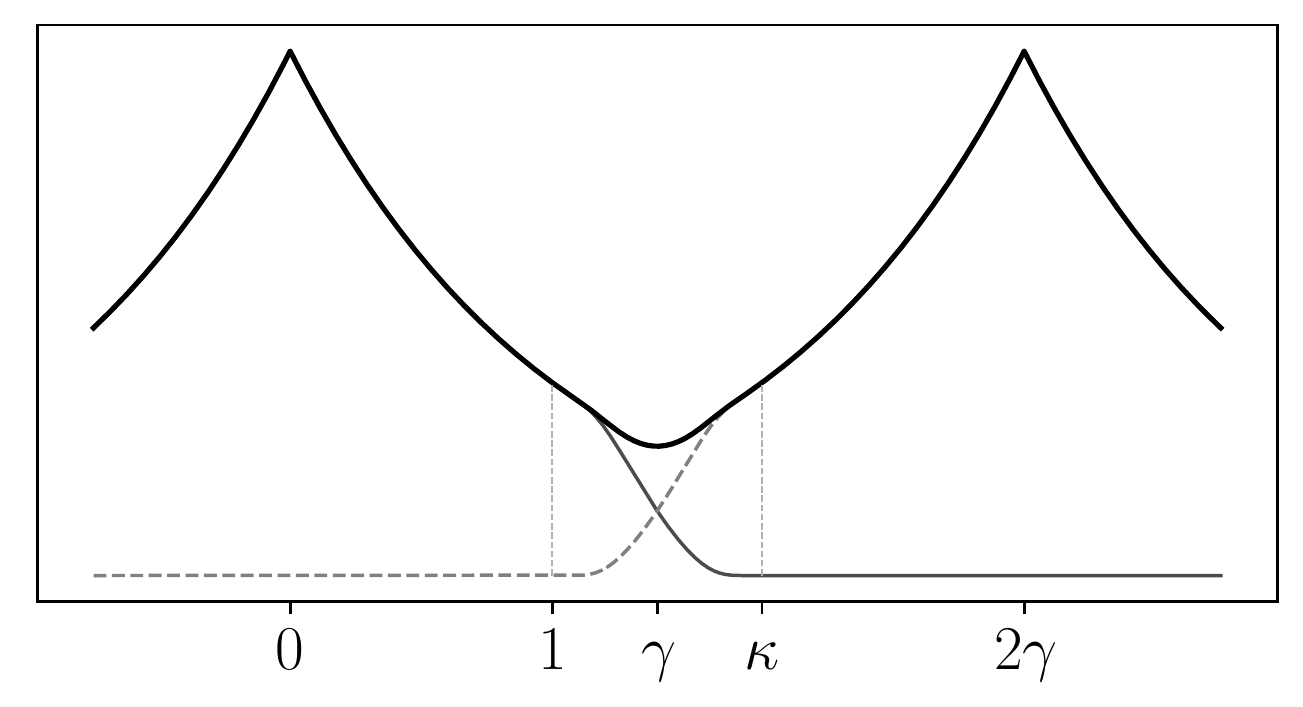}
\end{tabular}
\caption{Illustration of the two considered types of periodization for $d=1$ as in \eqref{period}, for the choices \eqref{ex1} in (a) and \eqref{ex2} in (b). The black graphs in each case correspond to $\rhoext$, shown for $\rho(x) = \exp(-\abs{x})$.}
\label{fig:periodization}
\end{figure}}

\rev{Returning to domains of general dimension $d$, %
  we first highlight our results concerning  Question I.}  In the case of   smooth periodization, 
it  was previously  shown in \cite[Thm.~2.3]{BCM}  for Mat\'ern
and various related \rev{(e.g., anisotropic Mat\'ern)} covariances that by
taking   $\gamma$  sufficiently large, one can always obtain a positive definite periodized covariance function, and hence a \emph{grid-independent periodic random field}. However, the  required size of $\gamma$ was not quantified \rev{in \cite{BCM}}.
In our first main result -- Theorem \ref{thm:smoothcond} below -- we show  that for Mat\'ern
covariances, %
for \emph{any} \rev{Mat\'ern smoothness parameter $\nu>0$ and correlation length $\lambda >0$,} it suffices to take
\begin{equation}\label{eq:condnew}
  \gamma \geq C \bigl( 1 + \lambda \max\{  \sqrt{\nu} (1 + \abs{\log \nu}), 1/\sqrt{\nu}\}  \bigr),
\end{equation}
\rev{with a constant $C>0$ that is independent of $\nu$ and $\lambda$.} \rev{Due to the existence of the periodic random 
field on the continuous level, for sampling on a discrete grid the result applies to any grid size $h$, defined as the (uniform) spacing between the sampling points $\bsx_i$.} 

\rev{This should be compared with the corresponding result for classical  periodization \cite{GKNSS}. There,
assuming in addition $\nu\geq \frac12$,}
a sufficient condition for positive definiteness of the form
\begin{equation}\label{eq:condgknss}
  \gamma \geq C \bigl( 1+\lambda \max\{ \sqrt{\nu} (1+ \abs{\log \nu}),  \sqrt{\nu}  \,|\!\log(h/\lambda)|\}  \bigr)
\end{equation}
\rev{was obtained} \rev{(again with $C$ independent of} \rev{$\nu$ and $\lambda$, as well as $h$)}. 
This bound was  seen to be essentially sharp in numerical experiments in \cite{GKNSS}; in particular, the required $\gamma$
indeed diverges as $h\to 0$, and one thus obtains, for any fixed $h>0$, a positive definite covariance matrix on a finite grid, but no underlying periodic random field.

\rev{For smooth periodization, we thus obtain in \eqref{eq:condnew}} essentially the same qualitative behaviour with respect to $\lambda$, $\nu$ as in \eqref{eq:condgknss},
but without the dependence on the grid size $h$, and including the regime $\nu \in (0, 1/2)$. 
\rev{The absence in \eqref{eq:condnew} of the logarithmic term
    in $h$ that was present in \eqref{eq:condgknss} leads to a
    substantial reduction of the computational complexity of the
    resulting sampling scheme. Moreover, the periodic extension of the original random field $Z$ on the continuous level can be used to obtain computationally attractive series expansions of $Z$ that provide alternatives to the classical KL expansion. These further conclusions are explained in detail in Section \ref{sec:conclusions}.}

The condition \eqref{eq:condnew} for all $\nu>0$ can also be regarded as a generalisation of the previous \rev{works} \cite{S02,GSPSJ06,MS18}, where periodization was considered with smooth truncations specifically designed for particular types of covariances, including the Mat\'ern case, but with the corresponding analysis limited to certain ranges of $\nu < 2$. Methods using smooth cutoff functions similar to the ones analysed here have also been tested computationally in \cite{HPA14,HKP}.

Our second main result concerns \rev{the answer to Question II in the case of non-smooth periodization}.  
 In the \MAT case it is well-known (see, e.g. \cite[Corollary 5]{GKNSSS},  \cite[eq.(64)]{BCM}) that the exact KL eigenvalues $\lambda_j$ of $Z$ \rev{in $L^2(D)$} decay with the rate
\begin{align}
  \label{kldecay} \lambda_j \ \leq \ C j^{-(1+ 2 \nu/d)}.  
\end{align} 
 Due to the preservation of covariance regularity  enjoyed by the smooth periodization,
 the ordered eigenvalues of $\Sigmaext$ in the smooth periodization case decay at the same
 optimal rate as in \eqref{kldecay}, see \cite[eq.\ (64)]{BCM}.     However, in the case of classical periodization, the associated  loss of regularity means that  the decay rate is not obvious.
 Supported by numerical evidence, it was conjectured in \cite{GKNSS} that under the
 condition \eqref{eq:condgknss}, the eigenvalues of $\Sigmaext$ 
 also exhibit, uniformly in $h$, the same asymptotic decay rate
 \eqref{kldecay} \rev{in the classical case}. 
 In Theorem \ref{thm:conjecture},  we prove this conjecture,
 up to a minor modification by a multiplicative factor of order $\mathcal{O}(|\!\log h|^\nu)$.

In summary, this paper provides  a complete characterisation of the performance of the two types of periodization in the case of Mat\'ern covariances. Both lead to optimal decay of covariance matrix eigenvalues, whereas the required size parameter $\gamma$ of the periodization cell is substantially more favorable in the smooth periodization case.

The outline of the paper is as follows: in Section \ref{sec:prelim}, we introduce some notions and basic results that are relevant to both the classical and smooth periodizations; in Section \ref{sec:ce}, we prove the conjecture
from \cite{GKNSS} (slightly modified) on the rate of eigenvalue decay for $\Sigmaext$ in the classical case; in Section \ref{sec:smooth}, we establish the quantitative condition \eqref{eq:condnew} for a smooth truncation; and in Section \ref{sec:num}, we illustrate our findings by numerical tests. \rev{In Section \ref{sec:conclusions}, we conclude with a discussion of the computational implications of our findings and of further applications.}

We use the following notational conventions: $\abs{\bx}$ is the Euclidean norm of $\bx\in \R^d$; $B_r(\bx)$ is the Euclidean ball of radius $r>0$ with center $\bx$. We use $C>0$ as a generic constant whose definition can change whenever it is used in a new inequality.

\section{Preliminaries}\label{sec:prelim}
\subsection{Fourier transforms} For a suitably regular function $v:\R^d \rightarrow \C$,  
the Fourier transform on $\R^d$ and its inverse are defined for
  $\bsomega,\bx \in \R^d$ by %
\begin{align} \label{eq:FT}
\widehat{v}(\bsomega) = \int_{\R^d} \exp(-\ri\bsomega\cdot  \bx)   v(\bx)\, \rd \bx,      \quad \text{and} \quad   {v}(\bx)
= \frac{1}{(2 \pi)^d} \int_{\R^d} \exp(\ri\bsomega \cdot \bx)   \widehat{v} (\bsomega ) \, \rd \bsomega .  
\end{align}

When $f:\R^d \rightarrow \C$ is $2 \gamma-$periodic in each coordinate direction and $f\in L_2\big([-\gamma,\gamma]^d\big)$
then $f$ can be represented as its Fourier series:
\begin{equation}
\label{eq:FS}
f(\bx) \ = \ (2 \gamma)^{-d} \sum_{\bk \in \bbZ^d} \widehat{f}_{\bk} \exp(\ri \, \bsomega_{\bsk}  \cdot \bsx ), 
\quad \text{where} \quad 
\widehat{f}_{\bk} \ = \ \int_{[-\gamma,\gamma]^d}
f( \bsx ) \exp\left(-  \ri \, \bsomega_\bsk  \cdot \bsx \right) \rd \bsx,
\end{equation} 
for all $\bk \in \bbZ^d$, with  $\bsomega_{\bk} := \pi \bk /\gamma $. Moreover, if $f$ belongs to a H\"older space $C^{0,\alpha}\big([-\gamma,\gamma]^d\big)$ for some $\alpha>0$ then the sum in \eqref{eq:FS} converges uniformly. 

Let $N \geq 2$ be an even integer,  set $h = 2 \gamma/N$
and introduce the infinite  uniform   
grid of points on $\R^d$:
\begin{equation} \label{eq:grid}
\bsx_{\bn} \,:=\,  \bn h   \qquad\text{for}\quad \bn \in \bbZ^d .
\end{equation}
By restricting $\bn$ to 
lie in \begin{align*}   {\overline{\bbZ}_N^d :=
	\big\{ -N/2, \ldots ,  N/2-1\big\}^d},
\end{align*}
we obtain a uniform grid on $[-\gamma,\gamma]^d$. 
The (appropriately scaled) discrete Fourier transform of the corresponding grid values of $f$, given by
\begin{align}\label{eq:eig_disc}
(S_N f)_{\bk} :=   h^d {\sum_{\bn \in \bbZbar_{N}^d}} f(\bx_{\bn})
\exp (- \ri \bsomega_{\bk} \cdot \bx_{\bn} ),
\quad\bk \in {\bbZbar}_{N}^d, 
\end{align}
yields an approximation of the Fourier coefficients $\widehat{f}_\bk$ for the same values of $\bk$. The approximation error can be quantified uniformly in $\bk\in \overline{\bbZ}_N^d$ by the following well-known result on the error of the trapezoidal rule applied to periodic functions, whose proof we include for convenience of the reader.

\begin{lemma}\label{lem:rect} Let $f$ be $2 \gamma-$periodic in each coordinate direction and $f\in C^{0,\alpha}\big([-\gamma,\gamma]^d\big)$ for some $\alpha>0$. Then 
	\begin{equation} \label{eq:trapezoidal}
	 h^d\! \sum_{\bn \in \oZ_N^d} f(\bx_{\bn})  - \int_{[-\gamma, \gamma]^d} f(\bx) \,\rd \bx  \ = \   \sum_{{0} \not= \bm \in \bbZ^d} \widehat{f}_{\bm N }  ,
	\end{equation}
	and in particular, for any $\bk \in {\bbZbar}_{N}^d$, 
	\begin{equation}\label{eq:samplingidentity}  (S_N f)_{\bk} - \widehat{f}_{\bk} = \sum_{{0} \not= \bm \in \bbZ^d} \widehat{f}_{\bk + \bm N} . 
	\end{equation}
	
\end{lemma}

\begin{proof}
	Using \eqref{eq:FS} and uniform convergence of the Fourier series by H\"older continuity of $f$,  we have 
	\begin{equation}
	\label{eq:L11}
	Q(f) :=  h^d \sum_{\bn \in \oZ_N^d} f(\bx_{\bn})  = \, (2 \gamma)^{-d} \sum_{\bm \in \bbZ^d} \widehat{f}_{\bm} \Bigg( h^d \sum_{\bn \in \oZ_N^d}
	\exp(\ri \bsomega_{\bm} \cdot \bx_{\bn} ) \Bigg).  
	\end{equation}
	Moreover,
	\begin{align*}
	\sum_{\bn \in \oZ_N^d} \exp(\ri \bsomega_{\bm} \cdot \bx_{\bn} ) & = \,  \sum_{\bn \in \oZ_N^d}\exp
	\left(\ri \frac{2 \pi}{N} {\bm} \cdot \bn \right) \\
	& = \ \sum_{\bn \in \oZ_N^d} \prod_{j=1}^d  \exp
	\left(\ri \frac{2 \pi}{N} {m_j} n_j \right)   = \, \prod_{j=1}^d \sum_{n \in \oZ_N}   \exp \left(\ri \frac{2 \pi}{N} {m_j} n \right)\,.
	\end{align*}  
	The last term vanishes unless  $m_j  = 0 (\mathrm{mod}\  N) $ for each $j = 1, \ldots , d$, in which case it takes the value $N^d$.  Since $h^d N^d = (2\gamma)^d$, we have 
	$
	Q(f) \ = \  \sum_{ \bm \in N \bbZ^d} \widehat{f}_{\bm} 
	$.
	Now \eqref{eq:trapezoidal} is obtained by noting that  $\int_{[-\gamma,\gamma]^d} f(\bx)\,\rd\bx =  \widehat{f}_{{0}}$.

	For fixed $\bk \in {\bbZbar}_{N}^d$, we introduce the function
	$ g (\bx) = f(\bx) \exp(- \ri \bsomega_{\bk} \cdot \bx)$, for $\bx \in [-\gamma, \gamma]^d$.
	Then $\widehat{f}_{\bk} = \int_{[-\gamma,\gamma]^d} g(\bx)\,\rd\bx$ and $(S_N f)_{\bk} = Q(g)$. From  \eqref{eq:trapezoidal} we conclude that
	\[
	   (S_N f)_{\bk} -  \widehat{f}_{\bk} \ = \ \sum_{{0} \not = \bm \in \bbZ^d}
	\widehat{g}_{\bm N}.\]
	Now \eqref{eq:samplingidentity} follows since $\widehat{g}_{\bm} = \widehat{f}_{\bk + \bm}$.
\end{proof} 

\subsection{Covariance functions} 

On a computational domain $D\subset\R^d$, we consider the fast evaluation of a Gaussian random field $Z(\bsx,
\omega)$ with
covariance function given by \eqref{eq:covar}.
In this paper we consider the important case of Mat\'ern covariance kernels with correlation length $\lambda>0$ and smoothness exponent $\nu>0$, given by  %
\begin{equation}\label{eq:materndef}
  \rho(\bx)  :=  \rho_{\lambda,\nu}(\bx):=\frac{2^{1-\nu}}{\Gamma(\nu)} \bigg(\frac{\sqrt{2\nu}|\bx|}{\lambda}\bigg)^{\nu}K_\nu\bigg(\frac{\sqrt{2\nu}|\bx|}{\lambda}\bigg)\, .  
\end{equation}
For its Fourier transform, we have (see, e.g., \cite[eq.\ (2.22)]{GKNSS}
\begin{align} \label{FT}
  \widehat \rho_{\lambda,\nu}(\bsomega) := \int_{\R^d} \rho_{\lambda,\nu}(\bx)\, \exp(-\ri \bsomega  \cdot \bx  ) \,\rd \bx  = C_{\lambda,\nu}  \bigg(\frac{2\nu}{\lambda^2}+|\bsomega|^2 \bigg)^{-(\nu+d/2)},\end{align}
where 
\begin{equation} \label{eq:C}
  C_{\lambda,\nu} := (2\sqrt{\pi})^d \frac{\Gamma(\nu+d/2)\,(2\nu)^\nu}{\Gamma(\nu)\,\lambda^{2\nu}}.
\end{equation}
The modified Bessel functions of the second kind $K_\nu$ have (see \cite[9.6.24]{AS}) the integral representations
\begin{equation}\label{Knuintegral}
  K_\nu(t) = \int_0^\infty e^{-t \cosh(s)} \cosh(\nu s)\,\dx s,
\end{equation}
which shows in particular that their values for fixed $t$ are monotonically increasing with respect to $\nu$ for $\nu \geq 0$, and also that $K_{-\nu} = K_\nu$. We will also use the following results that directly imply exponential decay of $\rho_{\lambda,\nu}(\bx)$ as $\abs{\bx}\to\infty$. Their proofs are given in Appendix \ref{auxproofs}.
\begin{lemma}\label{lem-estimate} Let  $\nu\geq 0$ and $t\geq 1/2$. Then  we have
	\begin{equation}\label{besselestimate}
	K_\nu(t) \leq e \frac{2^{2\nu}\Gamma(\nu)}{2\sqrt{2t}} e^{-t}\,.
	\end{equation}
\end{lemma}
\begin{lemma}\label{lem-diff} Let $n\in \N_0$. Then  $\frac{\dx^n }{\dx t^n}\big(t^\nu K_\nu(t)\big)$ is of the form 
	\begin{equation}\label{lem-2-eq}
	\frac{\dx^n }{\dx t^n}\big(t^\nu K_\nu(t)\big) = \sum_{j=0}^{\lfloor {n}/{2}\rfloor } a_{n,j}t^{\nu-j}K_{\nu-n+j}(t)\,,
	\end{equation} 
	with coefficients $a_{n,j}$ satisfying
	\[
	  \sum_{j=0}^{\lfloor {n}/{2}\rfloor } |a_{n,j}|\leq n!\,.
	  \]
\end{lemma}

Inspection of \eqref{eq:materndef} shows that changing $\lambda$ amounts to a rescaling of the computational domain $D$. Hence without loss of generality, we may assume our computational domain $D$ to be contained in the box $[-\frac12, \frac12]^d$, so that any difference of two points in $D$ is contained in $[-1,1]^d$. In what follows, we subsequently embed this box into a torus $[-\gamma,\gamma]^d$ with $\gamma \geq 1$, on which we define a $2\gamma$-periodic covariance $\rho^\ext$ such that $\rho^\ext(\bx) = \rho(\bx)$ for $\bx \in [-1,1]^d$, which means that the covariance between any pair of points in $D$ is preserved in replacing $\rho$ by $\rho^\ext$.

\section{Classical Periodization}\label{sec:ce}

In this section, we treat the classical periodization \eqref{ex1}. We prove in Theorem \ref{thm:conjecture} that for Mat\'ern covariances, the asymptotic decay of the eigenvalues of the extended matrix  $\Sigmaext$ is the same as that of the underlying KL eigenvalues \eqref{kldecay}, up to a    multiplicative   factor which grows logarithmically in $\abs{ \log (h)}^\nu $.  %
  This confirms a recent  conjecture \cite[eq.~(3.9)]{GKNSS}.

In this case  
  $\Rext$ is given by
\begin{equation}\label{eq:Rext_def}
 \Rext_{\bn, \bn'} \, = \, \rho^{\rm{ext}}\big(  \bx_{\bn} - \bx_{\bn'}   \big)   \,=\, \rho^{\rm{ext}}\big(  ({\bn - \bn'}) h  \big),
 \qquad \bn, \bn' \in \oZ_{N}^d ,
\end{equation}
with $\rhoext$ defined in \eqref{period}, \eqref{ex1}. 
$\Rext$ is  a circulant extension of the covariance matrix  $\Sigma$ of the form  \eqref{covmatrix},
obtained when
sampling  $Z$ at those points $\{\bx_{\bn}\}$ which lie  in  $[-1/2,1/2]^d$.
 Sampling on more  general $d$-dimensional rectangles can be treated in the same manner with the obvious modifications.
 If the index set  $\bn$ is   given  lexicographical ordering, then $\Sigma$
is  a nested block Toeplitz matrix where the number of nested levels is
the physical dimension $d$ and  $\Rext$ is a nested block circulant extension of it.

To analyse the eigenvalues of $\Rext$ it is useful to also consider 
the continuous
periodic covariance integral operator
\begin{equation*} %
  {\cR}^{\rm ext}\, v (\bsx)
  \,:=\, \ \int_{[-\gamma,\gamma]^d} \rho^{\rm ext} (\bsx-\bsxi)\, v(\bsxi) \,\rd\bsxi ,
  \qquad \bsx \in [-\gamma, \gamma]^d \,.
\end{equation*}
Then the scaled circulant matrix $h^d \Rext$ can be identified as a Nystr\"{o}m approximation of  
${\cR}^{\rm ext}$,  using the composite trapezoidal rule with respect to the uniform grid on  $[-\gamma, \gamma]^d$ given  by the points \eqref{eq:grid} with $\bn \in \oZ_N^d$.  
    The operator $\cRext$ is a compact operator on the space of
$2\len$-periodic continuous functions on $\bbR^d$, and so it has a
discrete spectrum with the only accumulation point at the origin.

The following result is standard (see for example \cite{GKNSS}).

\begin{proposition} \label{prop:eigs}
 
  \begin{itemize}
  \item[(i)]
The eigenvalues of $\cRext$  are 
$\widehat{\rho}_{\bsk},\,  \bk \in \bbZ^d$,
as defined in \eqref{eq:FS},  with  corresponding eigenfunctions normalized in $L^2([-\gamma, \gamma]^d)$:  
\[  
v_\bk(\bsx) = (2\gamma)^{-d/2} \exp(\ri \,
  \bsomega_\bsk \cdot \bsx) . 
\]

\item[(ii)]   The eigenvalues of   $h^d \Rext$ are $(S_N\rho)_{\bk}$, $\bk \in {\bbZbar}_{N}^d$, as defined in \eqref{eq:eig_disc},
with corresponding  eigenvectors normalized with respect to the Euclidean norm:
$$(\bsV_\bk)_{\bn} =  {N^{-d/2}} \exp(i \bsomega_{\bk} \cdot \bx_{\bn}),  \qquad \bn, \bk \in \oZ_N^d . $$

\end{itemize}
\end{proposition}

Note that $S_N \rhoext = S_N \rho$ since $\rho$ coincides with $\rhoext$ on $[-\gamma, \gamma]^d$.
It was convenient to introduce the scaling factor $h$ in (ii) above, since we can then identify  \eqref{eq:eig_disc}
as the trapezoidal rule approximation of the ($2 \gamma$-periodic) Fourier transform  defining $\widehat{\rho}_{\bsk} $.

Our analysis for estimating the decay of  $(S_N\rho)_{\bk}$
as $\vert \bk \vert \rightarrow \infty$
will be based on the formula \eqref{eq:samplingidentity} and  estimating the rate of decay of $\widehat{\rho}_{\bk}$.       
To this end, we now restrict our consideration to the Mat\'ern covariances, that is, for the remainder of this section we assume
\[
\rho = \rho_{\lambda,\nu}
\]
with some $\lambda,\nu>0$, where $\rho_{\lambda,\nu}$ is defined in \eqref{eq:materndef}.
In the recent paper \cite{GKNSS} the following theorem was proven for this case.

\begin{theorem}\label{thm:matern-growth}
Let $1/2 \leq \nu < \infty$, $\lambda \leq 1$, and $h/\lambda \leq e^{-1}$. Then there exist $C_1, C_2 >0$ which may depend on
 $d$ but are independent of $\gamma, h,
\lambda, \nu$, such that $\Rext$ is positive definite if
\begin{align}
\frac{\gamma}{\lambda}  \ \geq  \  C_1\  + \ C_2\,  \nu^{\frac12} \, \log\bigl( \max\big\{ {\lambda}/{h}, \, \nu^{\frac12}\big\} \bigr) \, .
\label{eq:alphahsmall}
\end{align}
\end{theorem}

We are now in a position to formulate our result on the rate of decay of eigenvalues of  $\Rext$  for Mat\' ern covariances,
which proves the conjecture made in \cite[eq.\ (3.9)]{GKNSS} up to an additional factor of order $|\!\log h|^\nu$ in the constant. 

  \begin{theorem}\label{thm:conjecture}
    Let $\nu$, $\lambda$, and $h$ be as in Theorem \ref{thm:matern-growth}. Let
    $\gamma^* = \gamma^*(\lambda,\nu,h)$ denote the smallest value of $\gamma\geq 1$ such that  condition
    \eqref{eq:alphahsmall} holds true and, adjusting $C_2$ if necessary,  assume that $C_2 > 2\sqrt{2}(2d-2)$.
  Suppose  $\gamma$ is chosen in the range  $\gamma^*\leq \gamma \leq a \gamma^*$ for some  $a\geq 1$
  independent of $h$ and let   $\Lambda^\ext_j$ denote   eigenvalues of $\Rext$ in  \rev{non-increasing}  order.    Then there exists $C>0$ such that for all $N$,
  	\begin{equation}\label{eq:conjecture}
  	 0 \ < \  \sqrt{\frac{\Lambda^\ext_j}{N^d}}  \  \leq\  C \, a^{\nu + d-1}\,  \left(\log \frac{\lambda}h \right)^\nu
          j^{-\frac12 - \frac\nu{d}},  \qquad j = 1,\ldots, N^d.
  	\end{equation}
      \end{theorem}
      
 The analysis which follows will be explicit in $j$, $h$, and $\gamma$ but not in the \MAT parameters $\lambda, \nu$ or in \rev{the} dimension $d$. \rev{Note that in applications, one is typically interested in the $h \ll \lambda \leq \operatorname{diam}(D)$, and thus our assumptions on $\lambda$ and $h$ only exclude practically irrelevant cases.} Recall that $C$ \rev{denotes} a generic constant which may change from line to line and may depend on $\lambda, \nu$ or $d$.
    
The proof of Theorem \ref{thm:conjecture} uses 
  Proposition \ref{prop:eigs}(ii),  which tells us that 
  $h^d\Lambda_j^\ext =  (S_N\rho)_{\bk(j)}$ for some $\bk(j) \in \oZ_N^d$.
Since Theorem \ref{thm:matern-growth} and Lemma \ref{lem:rect} give us 
\begin{equation}\label{eq:conjecturestart}
0 \ < \  (S_N\rho)_{\bsk}   
\ \leq\  |\widehat{\rho}_{\bsk} | + \bigabs{(S_N\rho)_{\bsk} - \widehat{\rho}_{\bsk}} 
= 
|\widehat{\rho}_{\bsk}| + \bigg|\sum_{{0} \not= \bm \in \bbZ^d} \widehat{\rho}_{\bk + \bm N}\bigg| \, , 
\end{equation}
the proof proceeds by obtaining   suitable estimates for the Fourier coefficients $\widehat{\rho}_\bk$ of the periodized covariance, as defined in \eqref{eq:FS}.  To this end, we use a cut-off function to isolate the artificial nonsmooth
part of $\rhoext$ created by the classical periodization.      
      We thus define an even smooth univariate cut-off function
  $\phi: \R \rightarrow \R $ by requiring that $\phi$ is supported on $[-3/4, 3/4]$ and 
  $$ \phi(t) = 1, \quad \text{for} \quad t \in [-1/2, 1/2], \quad \text{and} \quad \phi'(t) < 0, \quad t \in [1/2, 3/4].   $$
  For any $\gamma >0$, we can scale this to a cut-off function supported on $[-3\gamma/4, 3\gamma/4]$   \
  by  defining 
   $$ \phi_\gamma(\bx) = \phi\big(\vert \bx \vert /\gamma\big),\qquad \ \bx \in \R^d .  $$ 
      Using this function, we now write
      \begin{equation}\label{eq:cecutoff_decomp}
       \rhoext \  : = \ \beta + \sigma, \quad \text{where} \quad  \beta = \rhoext \phi_\gamma \quad \text{and} \quad   \sigma =  \rhoext (1 - \phi_\gamma) . 
      \end{equation}
            Thus  $\beta$ coincides with $\rho$ in a neighbourhood of the origin and vanishes in a neighbourhood of the interface 
        where $\rhoext$ has undergone  its (nonsmooth) extension,
        while the support of $\sigma$ covers exactly this interface. In the following two lemmas, we separately estimate $\widehat{\beta}_\bk$ and $\widehat{\sigma}_\bk$, defined
      as in the right-hand side of \eqref{eq:FS}.

      \begin{lemma}\label{lmmsmoothdecay}
      For $r \in \N$, there exists $C>0$ independent of $\gamma\geq 1$ such that
    \[
        \vert \widehat{\beta}_{\bk} \vert \leq C \Bigl(  \vert \widehat{\rho}(\bsomega_\bk) \vert \, +  \min\big\{ 1, \abs{\bsomega_{\bk}}^{-2r} \gamma^{-2r + d} \big\} \Bigr) \, , \quad \bk \in \oZ_N^d\,  ,  
      \]
      where $\widehat{\rho}$ is given by \eqref{FT}. 
        
      \end{lemma}
      
      \begin{proof}
       We have $\beta = \rho\phi_\gamma$, since $\phi_\gamma$ vanishes outside $[-\gamma, \gamma]^d$. By the convolution theorem,
    \begin{equation}\label{convolutionestimate}
    \begin{split} 
	 (2\pi)^{d} \big|\widehat\beta(\bsomega_\bk)\big| 
	 & = \abs{(\widehat\rho * \widehat{\phi}_\gamma)(\bsomega_\bk)} 
	 \\
	  &\leq \biggabs{ \int_{\abs{\bsxi}\leq \abs{\bsomega_\bk}/2} \widehat\rho(\bsxi) \widehat\phi_\gamma(\bsomega_\bk - \bsxi)\,\rd \bsxi } + \biggabs{ \int_{\abs{\bsxi}\geq \abs{\bsomega_\bk}/2} \widehat\rho(\bsxi) \widehat\phi_\gamma(\bsomega_\bk - \bsxi)\,\rd \bsxi }   .
	  \end{split}
\end{equation}
	The second term on the right can be estimated as
	  \begin{align}
	 \biggabs{ \int_{\abs{\bsxi}\geq \abs{\bsomega_\bk}/2} \widehat\rho(\bsxi) \widehat\phi_\gamma(\bsomega_\bk - \bsxi)\,\rd \bsxi } 
	 &  \leq \norm{ \widehat \phi_\gamma }_{L^1(\R^d)} \max_{\abs{\bsxi}\geq \abs{\bsomega_\bk}/2} \abs{\widehat\rho(\bsxi)} 
	 \nonumber \\
         & \leq C \norm{\widehat\phi_1}_{L^1(\R^d)} \abs{\widehat\rho(\bsomega_\bk/2)}
         \leq  C \abs{\widehat\rho(\bsomega_\bk) } \label{a}
	\end{align}
         with generic constants $C>0$ independent of $\gamma$, where we have used 
         \begin{equation}\label{L1invariance}
         \Vert \widehat{\phi}_\gamma\Vert_{L^1(\R^d)} = \Vert \widehat{\phi}_1\Vert_{L^1(\R^d)}, \qquad \gamma >0,
         \end{equation}
         and the decay properties of $\widehat\rho$ given in \eqref{FT}. 

          For the first term on the right of \eqref{convolutionestimate}, we use two separate bounds, suitable for  for small and large $\abs{\bsomega_\bk}$, to obtain 
          \[
             \biggabs{ \int_{\abs{\bsxi}\leq \abs{\bsomega_\bk}/2} \widehat\rho(\bsxi) \widehat\phi_\gamma(\bsomega_\bk - \bsxi)\,\rd \bsxi } \leq   \min\Bigl\{ \norm{\widehat\rho}_{L^\infty(\R^d)} \norm{\widehat\phi_\gamma}_{L^1(\R^d)}, \norm{ \widehat \rho }_{L^1(\R^d)} \max_{\abs{\bsxi}\geq \abs{\bsomega_\bk}/2} \abs{\widehat\phi_\gamma(\bsxi)} \Bigr\} .
          \]
          Now note that $\norm{\widehat\rho}_{L^\infty(\R^d)} \norm{\widehat\phi_\gamma}_{L^1(\R^d)}$ is uniformly bounded with respect to $\gamma$ by \eqref{L1invariance} and \eqref{FT}.
          Since $\phi_\gamma$ is smooth, integration by parts yields
            \[     \rev{     \widehat{\phi_\gamma}(\bsxi)    =  \abs{ \bsxi }^{-2r} \int_{\mathbb{R}^d}  \phi_\gamma(\bx)\,  (-\Delta_{\bx})^r\!
              \exp(- \ri \bsxi \cdot \bx) \, \rd \bx =  \abs{ \bsxi }^{-2r} \int_{\mathbb{R}^d} 
            \bigl[  (-\Delta)^r\! \phi_\gamma(\bx) \bigr] \, \exp(- \ri \bsxi \cdot\bx)\, \rd \bx     .   } 
              \] 
          Thus,  since $\widehat\rho \in L^1(\R^d)$, we have,    for any $r \in \N$,
          \begin{align}
           \norm{ \widehat \rho }_{L^1(\R^d)} \max_{\abs{\bsxi}\geq \abs{\bsomega_\bk}/2} \abs{\widehat\phi_\gamma(\bsxi)}  
            &  \leq C\abs{\bsomega_{\bk}}^{-2r} \int_{\R^d} \bigabs{ (-\Delta)^r \phi_\gamma(\bx)} \,\rd\bx \nonumber  \\
           &  \leq C\abs{\bsomega_{\bk}}^{-2r} \gamma^{-2r+ d}  \int_{\R^d} \bigabs{ (- \Delta)^r \phi_1(\bx)} \,\rd\bx 
             \leq   C\abs{\bsomega_{\bk}}^{-2r} \gamma^{-2r+ d} . 
         \label{b} \end{align}
Combining \eqref{a} and \eqref{b} with \eqref{convolutionestimate}   completes the proof.
         \end{proof}
         
For estimating $\abs{\widehat\sigma_\bk}$, we use the following auxiliary result, which is proved in the appendix.
         
         \begin{lemma}\label{lmm:sigmaest}
	Let $\bsalpha\in \Z^d$ with $|\bsalpha|_\infty \leq 2$ and $|\bx|\geq \gamma/2$ with $\gamma\geq 1 $. Then, with $\sigma$ as given in \eqref{eq:cecutoff_decomp}, there exists  $C$  independent of $\gamma$ such that
	\[
	  | \partial^{\bsalpha}\sigma(\bx)|
	   \leq  C e^{- \frac{\sqrt{2\nu}}{4{\lambda} }\gamma}  \bigg(\frac{\sqrt{2\nu}|\bx|}{\lambda}\bigg)^{\nu}e^{  -\frac{\sqrt{2\nu}|\bx|}{2\lambda}}\,,  \quad \text{for all} \quad \bx \in [-\gamma, \gamma]^d\backslash B_{\gamma/2}({0}).
	\] 
\end{lemma}

\begin{lemma}\label{lmmkinkdecay}
Let $1/2 \leq \nu < \infty$. Then there exists $C>0$ independent of $\gamma \geq 1$ such that
	\begin{align}\label{L}
          \abs{\widehat\sigma_\bk} \leq C   \exp\left(-{L}\gamma \right)  \prod_{i=1}^d \min\bigl\{1, \abs{\bsomega_{\bk,i}}^{-2} \bigr\} \,,\quad  {\text{for all}
             \quad \bk \in \oZ_N^d,  \quad \text{where} \quad  {L}:= \frac{\sqrt{2\nu}}{4 \lambda}} \,.
	\end{align} 
\end{lemma}
      
  \begin{proof}  First we bound $\sigma_{\bk}$ in terms of $\sigma$ and its derivatives.
  When $\max_i \abs{\bsomega_{\bk,i}} \leq 1$ we use the estimate
  \begin{equation}\label{sigma1}
  \begin{aligned} 
  \abs{\widehat\sigma_\bk } =\bigg| \int_{[-\gamma,\gamma]^d } \sigma(\bx) e^{-\ri \bsomega_\bk \cdot \bx}\,\rd\bx \bigg|
  &\leq \int_{[-\gamma,\gamma]^d } |\sigma(\bx) |\,\rd\bx \,.
  \end{aligned}
  \end{equation} 

 For the case where $\abs{\bsomega_{\bk,i}} > 1$ for at least one value of $i$, we integrate by parts dimensionwise to best exploit the limited smoothness of $\sigma$ across the boundary of $[-\gamma,\gamma]^d$. We assume without loss of generality that $\abs{\bsomega_{\bk,1}} > 1$ and we just give the proof for $d=2$ to simplify the exposition; higher dimensions are analogous.
 Integration by parts {in} \eqref{sigma1} {twice} with respect to $x_1$ gives
  \begin{align} 
&    | \widehat\sigma_\bk |  =  \abs{\bsomega_{\bk,1}}^{-1}\bigg| \int_{[-\gamma,\gamma]^2} \partial_{x_1} \sigma(\bx) e^{- \ri \bsomega_\bk \cdot \bx}\,\rd\bx \bigg|  \nonumber \\
     & \ \  =  \abs{\bsomega_{\bk,1}}^{-2} \bigg|\int_{-\gamma}^\gamma \biggl( \int_{-\gamma}^\gamma  \partial^2_{x_1} \sigma(\bx) e^{- \ri \bsomega_{\bk,1}  x_1}\,\rd x_1 
      - \bigl[ \partial_{x_1} \sigma(\bx)  e^{- \ri \bsomega_{\bk,1}  x_1} \bigr]_{x_1 = -\gamma}^\gamma \biggr)e^{-\ri \bsomega_{\bk,2}x_2}  \,\rd x_2 \bigg| .  \label{sigma2}
  \end{align}
  Now denoting
   \[
 \begin{aligned}
 \tilde \sigma_1(x_2) = \int_{-\gamma}^\gamma  \partial^2_{x_1} \sigma(\bx) e^{- \ri \bsomega_{\bk,1}  x_1}\,\rd x_1 
 - \bigl[ \partial_{x_1} \sigma(\bx)  e^{- \ri \bsomega_{\bk,1}  x_1} \bigr]_{x_1 = -\gamma}^\gamma \,  ,
 \end{aligned}
 \]
 we get $\tilde \sigma_1(\gamma)= \tilde \sigma_1(-\gamma)$ by the periodicity of $\sigma(\bx)$ as a function of $x_2$.  
 If, in addition,  $\abs{\bsomega_{\bk,2}} \geq 1$, then integrating by parts with respect to $x_2$  gives
  \begin{align} 
     \abs{\widehat\sigma_\bk }
      & =  \abs{\bsomega_{\bk,1}}^{-2}\abs{\bsomega_{\bk,2}}^{-2} \biggl|\int_{-\gamma}^\gamma \partial_{x_2}^2\tilde \sigma_1(x_2) e^{-\ri \bsomega_{\bk,2}x_2}\,\rd x_2 - \bigl[ \partial_{x_2} \tilde\sigma_1(x_2)  e^{-\ri \bsomega_{\bk,2}x_2} \bigr]_{x_2 = -\gamma}^\gamma \biggr| \nonumber \\
      & = \abs{\bsomega_{\bk,1}}^{-2}\abs{\bsomega_{\bk,2}}^{-2} \bigg| \int_{[-\gamma,\gamma]^2} \partial_{x_1}^2 \partial_{x_2}^2 \sigma(\bx) e^{-\ri \bsomega_\bk\cdot \bx}\,\rd\bx    - \biggl[ \int_{-\gamma}^{\gamma} \partial_{x_2} \partial^2_{x_1} \sigma(\bx) e^{-\ri \bsomega_\bk\cdot \bx} \,\rd x_1 \biggr]_{x_2=-\gamma}^\gamma
      \nonumber \\
        & \quad \ - \int_{-\gamma}^\gamma  \bigl[ \partial_{x_1}\partial_{x_2}^2 \sigma(\bx)  e^{-\ri \bsomega_\bk\cdot \bx} \bigr]_{x_1 = -\gamma}^\gamma \,\rd x_2 + \Bigl[ \bigl[ \partial_{x_1}\partial_{x_2} \sigma(\bx)  e^{-\ri \bsomega_\bk\cdot \bx} \bigr]_{x_1 = -\gamma}^\gamma  \Bigr]_{x_2 = -\gamma}^\gamma      
        \bigg|.
\label{sigma3} \end{align}
To estimate the right-hand sides of \eqref{sigma1}, \eqref{sigma2} and \eqref{sigma3},
we use Lemma \ref{lmm:sigmaest}, which directly gives the desired bound for the last term on the  right-hand side of \eqref{sigma3}. Moreover,
\[
\begin{split}
 \int_{[-\gamma,\gamma]^2}| \partial^{\bsalpha}\sigma(\bx)| \rd\bx  
&
\leq C
e^{- \frac{\sqrt{2\nu}}{4{\lambda}}\gamma}\int_{\abs{\bx} > \gamma/2} 	  \bigg(\frac{\sqrt{2\nu}|\bx|}{\lambda}\bigg)^{\nu}e^{  -\frac{\sqrt{2\nu}|\bx|}{2\lambda}}\, \rd \bx 
\leq C e^{- \frac{\sqrt{2\nu}}{4{\lambda}}\gamma}\,,
\end{split}
\]
which is the required bound for  
 \eqref{sigma1} and the first terms on the right-hand sides of \eqref{sigma2} and \eqref{sigma3}. Finally, for the second term on the right hand side of \eqref{sigma3} we have with $\bsalpha=(1,2)$
\[
\begin{split}
\biggl[ \int_{-\gamma}^{\gamma} \partial^{\bsalpha}\sigma ({\bx}) e^{-\ri \bsomega_\bk\cdot \bx} \,\,\rd x_1 \biggr]_{x_2=-\gamma}^\gamma
 &
 \leq  
 \bigg[ \int_{-\gamma}^{\gamma} \big|\partial^{\bsalpha}\sigma(\bx) \big|\,\rd x_1\bigg]_{x_2=\gamma} 
 + 
  \bigg[ \int_{-\gamma}^{\gamma} \big|\partial^{\bsalpha}\sigma(\bx) \big|\,\rd x_1\bigg]_{x_2=-\gamma}
\\
&
\leq 
C e^{- \frac{\sqrt{2\nu}}{4{\lambda} }\gamma} \int_{-\gamma}^{\gamma}  \bigg(\frac{\sqrt{2\nu(x_1^2+\gamma^2)}}{\lambda}\bigg)^{\nu}e^{  -\frac{\sqrt{2\nu(x_1^2+\gamma^2)} }{2\lambda}}\rd x_1
\\
&
\leq 
C  e^{  -\frac{\sqrt{2\nu}  }{4\lambda}\gamma}   \,.
\end{split}
\]
Carrying this out in the same way for further other terms we obtain the desired result for $d=2$.
      The proof for $d>2$ can be done   analogously, giving rise to $2^d$ terms in \eqref{sigma3}. 
  \end{proof}     

Based on these preparations, we now complete the proof of Theorem \ref{thm:conjecture}.

  \begin{proof}[Proof of Theorem \ref{thm:conjecture}.]
  We estimate the two terms on the right hand side of \eqref{eq:conjecturestart}.
  For the first term, combining Lemmas \ref{lmmsmoothdecay} and \ref{lmmkinkdecay} yields
 \begin{align} \label{E} 
   |\widehat{\rho}_{\bk} |\leq C \bigg( \abs{ \widehat{\rho}(\bsomega_\bk) } \, +  \min\big\{ 1, \abs{\bsomega_{\bk}}^{-2r} \gamma^{-2r + d} \big\} + e^{-{L}\gamma }  \prod_{i=1}^d \min\big\{1, \abs{\bsomega_{\bk,i}}^{-2} \big\} \bigg)
 \end{align} 
  with $C>0$ independent of $\gamma$ and $r \in \N $ arbitrary. We choose fixed  $r \geq \nu + d/2$. For the first term on the right of \eqref{E},
using \eqref{FT} with $\bsomega_\bk = \pi \bk /\gamma$ gives
  \[
      \abs{ \widehat{\rho}(\bsomega_\bk) } \leq C  \gamma^{2 \nu + d} ( 1+  \abs{\bk} )^{-(2\nu + d)} .
  \]
  Moreover, the second term of \eqref{E} can be estimated by
  \begin{align}\label{F}
    \min\big\{ 1, \abs{\bsomega_{\bk}}^{-2r} \gamma^{-2r + d} \big\}
      \leq  \gamma^d ( 1+ \abs{\bk})^{-2r} \leq  \gamma^{d}  ( 1+  \abs{\bk} )^{-(2\nu + d)}, \quad \bk \in \mathbb{Z}^d. 
    \end{align} 
To see the first inequality in \eqref{F}, consider $\bk = {0}$ and $\bk \not= {0}$ separately. In the latter case the inequality follows from the elementary estimate \rev{$\pi \vert \bk \vert \geq (1 + \vert \bk \vert$).}
  Estimating the remaining term of \eqref{E} in a similar way, we obtain
  \begin{equation}\label{rhohatestimate}
   | \widehat{\rho}_{\bk} | \leq  C \bigg( \gamma^{2 \nu + d} ( 1+  \abs{\bk} )^{-(2\nu + d)}
 + e^{-{L}\gamma } \gamma^{2d} \prod_{i=1}^d (1 +\abs{\bk_i})^{-2}  \bigg) \,.
\end{equation}
The second term on the right-hand side of \eqref{rhohatestimate} will turn out to be dominated by the first. But to finish the argument
  we also have to estimate the second term on the right-hand side of \eqref{eq:conjecturestart}.
  To do this note that 
  since $\max_{i=1,\ldots,d} \abs{\bk_i} \leq N/2$ for $\bk \in {\bbZbar}_{N}^d$, and $\bm \in \mathbb{Z}^d$, we have
  \[
  \big(1+|\bk+N\bm|\big)^{-1} \leq C \big(1+N|\bm|\big)^{-1}\quad\text{and}\quad   \big(1+|k_i+Nm_i|\big)^{-1} \leq C \big(1+N|m_i|\big)^{-1}\,,
  \]
  for $\bm\in \Z^d$ with $C$ independent of $\bk$ and $N$. Thus, we get from \eqref{rhohatestimate} 
  \begin{align}
  \bigg|  \sum_{{0} \neq \bm \in \bbZ^d} \widehat{\rho}_{\bk + N \bm }\bigg|
      &
       \leq C 
       \Bigg( {\gamma^{2 \nu + d}}\sum_{{0} \neq \bm \in \bbZ^d} ( 1 + N \abs{\bm})^{-(2\nu + d)}   
     \nonumber    \\
                & \quad\qquad
        +\ e^{-{L}\gamma } \gamma^{2d}  \sum_{{0} \neq \bm \in \bbZ^d}\prod_{i=1}^d  \big(1 + N\abs{m_i}\big)^{-2}  
        \Bigg) \,.
    \label{H} \end{align}
Now, by elementary arguments,
\[
\begin{split} 
\sum_{{0} \neq \bm \in \bbZ^d}\prod_{i=1}^d  \big(1 + N\abs{m_i}\big)^{-2} 
&  =  \sum_{  \bm \in \bbZ^d}\prod_{i=1}^d  \big(1 + N\abs{m_i}\big)^{-2} -1
\\
  &  =     \bigg( \sum_{m \in \Z} \big(1 + N\abs{m}\big)^{-2}\bigg)^d  - 1
           \leq   \Bigl( 1 + \frac{\pi^2}{3} N^{-2}\Bigr)^d - 1  . 
 \end{split}
        \] 
 Inserting this into the second term of \eqref{H}, and estimating the first term similarly, we obtain
  \begin{equation} \label{rhohatestimate-2}
 \begin{aligned}    
\bigg|  \sum_{{0} \neq \bm \in \bbZ^d} \widehat{\rho}_{\bk + N \bm }\bigg|   
      & \leq C \Big( \gamma^{2 \nu + d}N^{-(2\nu + d)} + e^{-{L}\gamma} \gamma^{2d} N^{-2} \Big) \,.
  \end{aligned}
\end{equation}

We now  show that the first  term in  \eqref{rhohatestimate}  is dominant in both estimates   \eqref{rhohatestimate}, \eqref{rhohatestimate-2}.
First note  that, by elementary arguments,  
  $
    \prod_{i=1}^d (1 +\abs{\bk_i})^{2} \ \geq \ 1 + \vert \bk \vert^2 \ \geq \ \frac12 (1 + \vert \bk \vert)^2 $.
Then, for $\bk \in \oZ_N^d$,  we have  
  \begin{align}\label{61}
    \bigg(\prod_{i=1}^d (1 +\abs{\bk_i})^{-2} \bigg)  \left(1 + \vert \bk\vert \right)^{2 \nu + d}
    \ \leq \ C  \left(1 + \vert \bk\vert \right)^{2 \nu + d-2} \ \leq \ C N^{2 \nu + d-2}
  \end{align}
Now by choice of $\gamma$ and \eqref{eq:alphahsmall}, we have $\gamma \, \geq\, \gamma^* \, > \,  C_2 \lambda \sqrt{\nu} \log(\lambda/h)$, and using the definition of $L$ in  \eqref{L}, we obtain \begin{align} \label{62}  L \gamma \ \geq \ \frac{C_2\nu }{2\sqrt{2}}  \log(\lambda/h) \ = \
  \log((\lambda/h)^{C_2 \nu/2 \sqrt{2}})\, . 
\end{align} 
Then, combining \eqref{61} and \eqref{62} and recalling that  $h = 2\gamma/N$,  we obtain  
\begin{align}   e^{-L\gamma} \gamma^{2d} \prod_{i=1}^d (1 +\abs{\bk_i})^{-2} \
  & \leq\  C \left[h^{C_2 \nu/2\sqrt{2}} \, \gamma^{d - 2\nu}\,  N^{2 \nu + d -2} \right] \, \gamma^{2 \nu + d} (1 + \vert \bk \vert)^{-(2 \nu + d)}\nonumber  \\  & = C \left[h^{C_2 \nu/2\sqrt{2}- 2\nu -d +2 } \, \gamma^{2d-2} \right] \, \gamma^{2 \nu + d} (1 + \vert \bk \vert)^{-(2 \nu + d)} . \label{63}
\end{align}
By choice of $C_2$, and using $\nu \geq 1/2$,   we have $C_2 > 2 \sqrt{2}  ( 2 + (d-2)/\nu)  $ and so the exponent of $h$ in \eqref{63} is positive. Also since $\gamma \leq a \gamma^*$, $\gamma$ grows at most logarithmically in $h$ with a multiplicative constant which grows at most linearly in $a$. This yields a bound on the second term in \eqref{rhohatestimate} and thus
\begin{align*} 
  \vert \widehat{\rho}_{\bk} \vert \ \leq \ C \, a^{2d -2}\,  \gamma^{2 \nu + d}(1+ \vert \bk \vert)^{-(2 \nu + d)}. 
\end{align*}
Turning to the second term on the right-hand side of \eqref{rhohatestimate-2} we obtain,  similarly,
\begin{equation} \label{64} 
\begin{aligned}
  e^{-L\gamma} \gamma^{2d} N^{-2} \ &\leq \ C  \left[h^{C_2 \nu/2 \sqrt{2}} \,  \gamma^{d- 2 \nu } \, N^{2 \nu + d -2}\right] \,    \gamma^{2 \nu + d} \, N^{-(2 \nu + d)}   \\
   &\leq \ C  \, a^{2d-2} \, \gamma^{2 \nu + d} \, (1 + \vert \bk \vert)^{-(2 \nu + d)},  
   \end{aligned}
   \end{equation}
when $\bk \in \oZ_N^d$. 
Inserting \eqref{63} and \eqref{64} into \eqref{rhohatestimate} and \eqref{rhohatestimate-2}, we obtain
 \begin{align} \label{65}
   (S_N\rho)_{\bsk}  \leq C\,  a^{2d-2} \, \gamma^{2\nu+ d }\,  \big( 1+  \abs{\bk} \big)^{-(2\nu + d)}, \quad \bk \in {\bbZbar}_{N}^d.
   \end{align}

Now, to finish the proof, let $\{\abs{\bk^*_j}: \ j = 1,\ldots,N^d\}$, be a  non-decreasing ordering of
  the numbers $\{\abs{\bk}: \, \bk \in \oZ_N^d\}$. As shown in \cite[Theorem 3.4]{GKNSS}, $\abs{\bk^*_j}$ is then  uniformly proportional to $j^{1/d}$, with constants that depend only on $d$.  Thus by \eqref{65}, 
  \begin{equation}\label{eq:estordered}
     (2\gamma)^{-d} (S_N\rho)_{\bsk^*_j}  \leq C a^{2d-2} \, \gamma^{2\nu} \, j^{-(2\nu/d + 1)}, \quad  j = 1,\ldots,N^d.
   \end{equation}
   Now by the hypothesis of the theorem, the numbers $\lambda^*_j := N^{-d} \Lambda^\ext_j$, $j = 1, \ldots, N^d$, are non-increasing and, by Proposition \ref{prop:eigs}(ii), provide a non-increasing ordering of the values  $(2\gamma)^{-d} (S_N\rho)_{\bsk}$,  $\bk \in {\bbZbar}_{N}^d$.  Then we claim that,  for any integer $0<J<N^d$,
  \begin{align} \label{66}
      \sum_{j > J} \lambda^*_j \ \leq \  C \, a^{2d-2} \, \sum_{n> J}  \gamma^{2\nu} n^{-(2\nu/d + 1)},
    \end{align} 
    with $C$ as in \eqref{eq:estordered}. If this were not true then, by \eqref{eq:estordered}, and for some $J$, 
    \[\sum_{j > J} \lambda^*_j  \ > \ \sum_{n>J}   (2\gamma)^{-d} (S_N\rho)_{\bsk^*_n}\, .   \]
    Since the terms in the right-hand sum also provide an ordering for the eigenvalues $\lambda_j^*$ this  
    contradicts the assumed non-increasing property of $\lambda^*_j$. As a consequence of \eqref{66}, we then have
  \[
  \begin{aligned}
   N^{-d} \Lambda^\ext_j   \ = \  
   \lambda^*_j  & \ \leq \ \frac{2}{j} \sum_{n = \floor{j/2} + 1 }^{j} \lambda^*_n \ \leq\  \frac{2}{j} \sum_{n>\floor{j/2}} \lambda^*_n 
       \\
       & \leq  \ C \, a^{2d-2}\, \gamma^{2\nu}\,   \frac{2}{j}\,   \sum_{n>\floor{j/2}}  n^{-(2\nu/d + 1)} \ \leq\  C\,  a^{2d-2} \,  \gamma^{2\nu} \,  j^{-(2\nu/d + 1)} .
  \end{aligned}
\]
With the condition \eqref{eq:alphahsmall} on $\gamma$,
implying that  $\gamma \leq C a \log(\lambda/h)$,  this completes the proof.
  \end{proof}

\section{Smooth Periodization}\label{sec:smooth}

We now establish a sufficient criterion on the periodization cell size $\gamma$ ensuring a positive definite periodic covariance function in the case of periodization with a smooth cutoff function as in \eqref{ex2}.
This amounts to proving a quantitative version of \cite[Theorem 2.3]{BCM} for the case of Gaussian random fields with Mat\'ern covariance as in \eqref{eq:materndef}.

We explicitly construct a suitable even cutoff function $\varphi_\kappa$ which vanishes outside $B_{\kappa}({0})$ such that $\varphi_\kappa = 1$ on $[-1,1]^d$ and  $\rho_{\lambda,\nu}\varphi_\kappa$ is a positive definite function. In this case, for $\gamma$ sufficiently large, there exists a periodic Gaussian random field on the torus $[-\gamma, \gamma]^d$ with the periodized covariance kernel
\begin{equation}\label{eq:smperiod}
   \rho^\ext(\bx) = \sum_{\bn \in \Z^d} (\rho_{\lambda,\nu}\varphi_\kappa)(\bx + 2\gamma \bn),
\end{equation}
such that $\rho^\ext = \rho_{\lambda,\nu}$ on $[-1,1]^d$, so that the corresponding random fields have the same law on the domain of interest contained in $[-\frac12,\frac12]^d$.
As shown in \cite[\S 3]{BCM}, if $\varphi_\kappa$ is sufficiently smooth, the eigenvalues of the covariance operator of the periodized random field then have the same asymptotic decay as those of the corresponding Mat\'ern covariance operator.

The cutoff function defined here is different from $\phi$ used in Section \ref{sec:ce}, which served only as a tool in the proof of Theorem \ref{thm:conjecture}. By contrast, the \rev{cutoff function $\varphi_\kappa$ on $\R^d$ derived from a univariate cutoff function $\varphi$} constructed here is used numerically in the computation of the random field. \rev{In order to cover the full range of Mat\'ern smoothness parameters $\nu>0$ in our analytical results, we need precise control of high-order derivatives of $\varphi$. Specifically, for $p:= \lceil \nu +\frac{d}{2}\rceil$ we require bounds of the form
\begin{equation}\label{eq:cutoffest}
  \sup_{t\in \R}\big|\varphi^{(\alpha)}(t)\big| \leq c_1 \bigg(\frac{c_2 p}{\kappa}\bigg)^\alpha, \quad
    \alpha = 0,\ldots, 2p,
\end{equation}
with some $c_1, c_2 > 0$. The use of such $\varphi$ mainly allows us to circumvent some further major technicalities in our proofs, and as the numerical tests in Section \ref{sec:num} show, one still observes similar results for cutoff functions for which no bound of the form \eqref{eq:cutoffest} is available.}

\rev{Our concrete choice of $\varphi$ is as follows:} let $N_P$ be the B-spline function with nodes $\{- P,\ldots,-1,0 \}$, where $P:=2p+1$. For $\kappa>0$ we define the even function $\varphi\in C^{2p}(\R)$ by
\begin{equation}\label{eq:bspline1}
\varphi(t)=\begin{cases}
	1 & \text{if}\ \  |t|\leq \kappa/2\\[3pt]
	\displaystyle \frac{2 P}{\kappa}\int_{-\infty}^{t+\kappa/2} N_{P}\biggl(\frac{2 P}{\kappa}\xi\biggr)\dx\xi & \text{if}\ \ t\leq -\kappa/2\,.
\end{cases}
\end{equation}
It is easy to see that $\varphi(t)=0$ if $|t|\geq \kappa $. 
\rev{This choice of $\varphi$ provides us with explicit bounds of the form \eqref{eq:cutoffest} on all required derivatives.}
From $N'_{r+1}(t)=N_r(t)-N_r(t-1)$ we infer that for $0\leq \alpha \leq 2p$,
\begin{equation}\label{k--01}
\sup_{t\in \R}\big|\varphi^{(\alpha)}(t)\big| \leq 2^{\alpha} \bigg(\frac{2 P}{\kappa}\bigg)^\alpha \,.
\end{equation}
We now define 
\begin{equation}\label{eq:bsplined}
 \varphi_\kappa(\bx):=\varphi(|\bx|)\qquad \text{and} \qquad 
\theta_\kappa = 1-\varphi_\kappa,\qquad \bx\in \R^d.
\end{equation}
With this choice of $\varphi_\kappa$ in \eqref{eq:smperiod}, we have $\rho^\ext = \rho_{\lambda,\nu}$ on $[-1,1]^d$ provided that 
\begin{equation}\label{eq:++}
 \gamma \geq \frac{\kappa + \sqrt{d}}{2},
\end{equation}
which reduces to the condition in \eqref{cutoff} when $d=1$.

\rev{In our following main result, we pursue} the basic strategy of \cite[Theorem 2.3]{BCM} to establish sufficient conditions in terms of $\nu,\lambda$ on the required value of $\kappa >0$ such that
\begin{equation} \label{k-00}
\widehat{\rho_{\lambda,\nu}\varphi_\kappa}(\bsomega) = \widehat{\rho_{\lambda,\nu}}(\bsomega) -\widehat{\rho_{\lambda,\nu}\theta_\kappa}(\bsomega)  >0\,,\quad \bsomega \in \R^d.
\end{equation}
\rev{In \cite{BCM}, for a more general class of covariance functions than considered here, only the existence of such $\kappa$ is established without further information on its size. The proof uses the integrability of derivatives of $\rho_{\lambda,\nu}\theta_\kappa$ to show that $\widehat{\rho_{\lambda,\nu}\theta_\kappa}$ decays at least as fast as $\widehat{\rho_{\lambda,\nu}}$, and then uses the exponential spatial decay of these derivatives to show that $\widehat{\rho_{\lambda,\nu}\theta_\kappa}$ can indeed be bounded by $\widehat{\rho_{\lambda,\nu}}$ if $\kappa$ is chosen sufficiently large. 
Here, we follow the same basic strategy, but extract information on the required size of $\kappa$. This needs detailed information on higher-order derivatives of $\rho_{\lambda,\nu}$ and $\theta_\kappa$, where the order increases with the value of $\nu$. The proof of Theorem \ref{thm:matern-growth} for the classical periodization relies in a similar manner on using spatial decay of $\rho_{\lambda,\nu}$ to control a perturbation term, but does not require derivative information.}

\begin{theorem}\label{thm:smoothcond}
	For $d\in\{1,2,3\}$ and $\varphi_\kappa$ as defined above, there exist constants $C_1, C_2$ such that for any $0<\lambda,\nu<\infty$, we have $\widehat{\rho_{\lambda,\nu} \varphi_\kappa} > 0$ provided that $\kappa > 1$ and
	\begin{equation}\label{kappacondition}
	 	  \frac{\kappa}\lambda \geq C_1 + C_2 \max\Big\{\nu^{\frac12} ( 1 + \abs{\ln \nu}) , \nu^{-\frac12} \Big\}.
	\end{equation}
\end{theorem}

Note that the condition \eqref{kappacondition}, together with \eqref{eq:++}, is similar to the one in Theorem \ref{thm:matern-growth}. There are two key differences: the restriction $\nu\geq \frac12$ does not appear, and since there is no discretization involved in Theorem \ref{thm:smoothcond}, there is no dependence on a grid size $h$ as in Theorem \ref{thm:matern-growth}. The restriction to the dimensionalities $d\in\{1,2,3\}$ that are relevant in applications is not essential, but allows us to avoid some further technicalities in the proof.

\begin{proof}
Note that since $ \widehat{\rho_{\lambda,\nu}\varphi_\kappa} (\bsomega) = \lambda^d \widehat{\rho_{1,\nu} \varphi_{\frac{\kappa}{\lambda}}}(\lambda\bsomega)$ for any $\lambda>0$,
if \eqref{k-00} holds with some $\kappa_1$ for $\rho_{1,\nu}$, then it also holds with $\kappa:=\lambda\kappa_1$ for $\rho_{\lambda,\nu}$. Consequently, it suffices to consider the case $\lambda = 1$ in what follows, and we write $\rho = \rho_{1,\nu}$ and $\theta = \theta_\kappa$. 
We consider the cases $  \frac{1}{2}\leq \nu<\infty$ and $0<\nu<\frac{1}{2}$ separately.

\medskip

\noindent\emph{Step 1. The case $\nu \geq \frac12$.}
With $r=|\bsomega|$, for the Fourier transform of the radial function $\rho\theta$ we have
\[
\begin{split} 
	\widehat{\rho\theta}(\bsomega)&=(2\pi)^{\frac{d}{2}}\int_{\kappa/2}^\infty [\rho\theta](t)(r t)^{-\frac{d}{2}+1}J_{\frac{d}{2}-1}(rt)t^{d-1}\dx t ,
\end{split}
\]
where $J_\alpha$ is the classical Bessel function of order $\alpha$. 
Now condition \eqref{k-00} with $\lambda=1$ is equivalent to
\begin{equation} \label{eq:eq}
C_{1,\nu}  \big( 2\nu +r^2 \big)^{-(\nu+\frac{d}{2})} \geq  (2\pi)^{\frac{d}{2}} \int_{\kappa/2}^\infty [\rho\theta](t)(r t)^{-\frac{d}{2}+1}J_{\frac{d}{2}-1}(rt)t^{d-1}\dx t, 
\end{equation}
where $C_{1,\nu}$ is given in \eqref{eq:C}. Since $\nu\geq \frac{1}{2}$ and $p\geq \nu+\frac{d}{2}$, it thus suffices to choose $\kappa$ such that
\begin{equation}\label{condlambda1}
C_{1,\nu} \geq (2\pi)^{\frac{d}{2}}\big(2\nu+r^2\big)^p \int_{\kappa/2}^\infty [\rho\theta](t)(r t)^{-\frac{d}{2}+1}J_{\frac{d}{2}-1}(rt)t^{d-1}\dx t,\quad \text{for all}\quad r \geq 0.
\end{equation}
 In what follows, as a consequence of \eqref{kappacondition} we can assume without loss of generality that
\[
\kappa   \frac{ \sqrt{2\nu}}{2\lambda}  =  \kappa\sqrt{\frac{\nu}{2}} \geq 2 \max\{ P, \nu d_1\},
\]
with $d_1=1+2(d-1)$. We now proceed to estimate
\[
\begin{split} 
	A(r) &:= (2\pi)^{\frac{d}{2}}\big(2\nu+r^2\big)^p \int_{\kappa/2}^\infty [\rho\theta](t)(r t)^{-\frac{d}{2}+1}J_{\frac{d}{2}-1}(rt)t^{d-1}\dx t\\
	&= (2\pi)^{\frac{d}{2}}\bigg[\sum_{\ell=0}^{p}\binom{p}{\ell}\big( 2\nu\big)^{p-\ell}r^{2\ell}\bigg] \int_{\kappa/2}^\infty [\rho\theta](t)(r t)^{-\frac{d}{2}+1}J_{\frac{d}{2}-1}(rt)t^{d-1}\dx t \\
	&= (2\pi)^{\frac{d}{2}}\sum_{\ell=0}^{p}\binom{p}{\ell} A_\ell (r),
\end{split}
\]
with
\begin{equation}\label{eq:al}
 A_\ell (r) : = \big( 2\nu\big)^{p-\ell}r^{2\ell} \int_{{\kappa}/{2}}^\infty [\rho\theta](t)(r t)^{-\frac{d}{2}+1}J_{\frac{d}{2}-1}(rt)t^{d-1}\dx t .
 \end{equation}
We consider the case $\ell\geq 1$. Using $\frac{\dx}{\dx z}\big[z^{\alpha }J_{\alpha }(z)\big]=z^{\alpha }J_{\alpha -1}(z)$, see \cite[page 45]{Wat}, we can write 
\[
 (r t)^{-\frac{d}{2}+1}J_{\frac{d}{2}-1}(rt)t^{d-1} =r^{-d} \frac{\dx}{\dx t} \Big[ (rt)^{\frac{d}{2}}J_{\frac{d}{2}}(rt)\Big]  \,.
\]
Integrating by parts, we get from   $[\rho\theta](\frac{\kappa}{2})=0$ and the exponential decay of the Mat\'ern covariance function
\[
\begin{split}
 A_\ell (r)
 &
  = \Big[\big( 2\nu\big)^{p-\ell}r^{-d} [\rho \theta](t) (rt)^{\frac{d}{2}}J_{\frac{d}{2}}(rt)\Big]_{\kappa/2}^\infty - \big( 2\nu\big)^{p-\ell}r^{2\ell} \int_{\kappa/2}^\infty  r^{-\frac{d}{2}} t^{\frac{d}{2}}J_{\frac{d}{2}}(rt) [\rho\theta]'(t) \dx t
 \\
 &
 = - \big( 2\nu\big)^{p-\ell}r^{2\ell} \int_{\kappa/2}^\infty  r^{-\frac{d}{2}} t^{\frac{d}{2}}J_{\frac{d}{2}}(rt) [\rho\theta]'(t) \dx t
 \\
 &
 = - \big( 2\nu\big)^{p-\ell}r^{2\ell} \int_{\kappa/2}^\infty  r^{-\frac{d}{2}} t^{1-\frac{d}{2}}J_{\frac{d}{2}}(rt) [\rho\theta]'(t) t^{d-1} \dx t
 \\
 &
 =  \big( 2\nu\big)^{p-\ell}r^{2\ell-2} \int_{\kappa/2}^\infty   [\rho\theta]'(t) t^{d-1}\frac{\dx }{\dx t} \Big[ (rt)^{1-\frac{d}{2}}J_{\frac{d}{2}-1}(rt)\Big]\dx t\,,
 \end{split}
\]
where in the last step we use $\frac{\dx}{\dx z}\big[z^{-\alpha }J_{\alpha }(z)\big]=-z^{-\alpha }J_{\alpha +1}(z)$ (see \cite[page 45]{Wat}).  
Integrating by parts again we arrive at 
\[
\begin{split}
A_\ell (r)
= - \big( 2\nu\big)^{p-\ell}r^{2\ell-2} \int_{\kappa/2}^\infty   (rt)^{1-\frac{d}{2}}J_{\frac{d}{2}-1}(rt) \big([\rho\theta]'(t) t^{d-1}\big)'\dx t\,.
\end{split}
\]
Repeating this argument we conclude that
\[
\begin{split}
	A_\ell (r) =(-1)^\ell \big( 2\nu\big)^{p-\ell} \int_{\kappa/2}^\infty (r t)^{1-\frac{d}{2}}J_{\frac{d}{2}-1}(rt)\big(\big(\big(\big([\rho\theta]'(t)t^{d-1}\big)'t^{1-d}\big)'\ldots\big)'t^{d-1} \big)'\dx t \,,
\end{split}
\]
where derivatives are taken $2\ell$ times. Employing Lommel's expression of $J_\alpha$, see \cite[page 47]{Wat},
\[
 J_\alpha (z)= \frac{(z/2)^{\alpha}}{\Gamma(1/2)\Gamma(\alpha+1/2)} \int_0^\pi \cos(z\cos \beta)\sin^{2\alpha} \beta\, \dx \beta,\qquad \alpha> -1/2,
\]
and 
$
J_{-1/2}(z)= \sqrt{\frac{2}{\pi}}\frac{\cos z}{\sqrt{z}}
$,
we can bound
\begin{equation} \label{eq:J}
|z^{-\alpha}J_\alpha (z)|\leq  \frac{(1/2)^{\alpha}}{\Gamma(1/2)\Gamma(\alpha+1/2)} \int_0^\pi  \sin^{2\alpha} \beta\, \dx \beta = \frac{(1/2)^{\alpha}}{\Gamma(\alpha+1)} ,\qquad \alpha> -1/2\,,
\end{equation}
where in the last equality we have used the relation between Gamma and Beta functions, see \cite[Section 6.2]{AS}.
Consequently
\[
\begin{split}
	 |A_\ell (r) | &\leq C_0 \big( 2\nu\big)^{p-\ell} \int_{\kappa/2}^\infty  \big|\big(\big(\big(\big([\rho\theta]'(t)t^{d-1}\big)'t^{1-d}\big)' \ldots\big)'t^{d-1}\big)'\big|\dx t \,,
\end{split}
\]
with $C_0=  \frac{(1/2)^{d/2-1}}{\Gamma(d/2)}$ for  $d\in \{1,2,3\} $. 

To finish the proof we need a technical lemma. For $\ell\in \N$ and $f$ having derivatives of sufficiently high order, we denote 
\[
\begin{split}
B_{ 2\ell ,d}(f,t)&:=\big(\big(\big(\big(f'(t)t^{d-1}\big)'t^{1-d}\big)' \ldots\big)'t^{d-1}\big)'\,;\qquad (2\ell\ \text{times})\\
B_{ 2\ell+1 ,d}(f,t)&:=\big(\big(\big(\big(f'(t)t^{d-1}\big)'t^{1-d}\big)' \ldots\big)'t^{1-d}\big)'\,;\qquad (2\ell+1\ \text{times})\,.
\end{split}
\]
\begin{lemma}\label{derivative} For $\ell\geq 1$ and $f$ having $2\ell$-th derivative, the term $B_{ 2\ell ,d}(f,t)$ has the form
	\begin{equation}\label{k-001}
	B_{ 2\ell ,d}(f,t)=\sum_{\alpha =1}^{ 2\ell } a_{ 2\ell ,\alpha } f^{(\alpha )}(t)t^{d-1+\alpha - 2\ell }
	\end{equation}
	with $B_{ 2\ell ,1}(f,t)=f^{( 2\ell )}(t)$, $B_{ 2\ell ,3}(f,t)= f^{( 2\ell )}(t)t^{2}+ 2\ell f^{( 2\ell -1)}(t)t$, 
	and when $d=2$ we have 
	\[
	\sum_{\alpha =1}^{ 2\ell } |a_{ 2\ell ,\alpha }|\leq 4^{\ell-1} 2 [(\ell-1)!]^2\,.
	\]
\end{lemma}
The proof of this lemma is given in Appendix \ref{auxproofs}.  We continue the proof of Theorem \ref{thm:smoothcond} by using the above lemma to obtain the estimate
\begin{equation}\label{eq:Al-lmm}
\begin{split} 
	 | A_\ell  (r) | 
	& \leq C_0 4^{\ell}  (\ell!)^2 \big( 2\nu\big)^{p-\ell} \max_{0<\alpha\leq 2\ell}\int_{\kappa/2}^{\infty}\bigg|  \frac{[\rho\theta]^{(\alpha)}(t)}{t^{2\ell-\alpha-d+1}}\bigg|\dx t
	\,.
\end{split}
\end{equation}
Thus, from \eqref{k--01} we get
\begin{equation}\label{k-03-1}
\begin{split} 
	\int_{\kappa/2}^{\infty}\bigg|  \frac{[\rho\theta]^{(\alpha)}(t)}{t^{2\ell-\alpha-d+1}}\bigg|\dx t & \leq \sum_{n=0}^\alpha \binom{\alpha}{n} \int_{\kappa/2}^{\infty}  \frac{ | \rho^{(n)}(t) \theta^{(\alpha-n)}(t)  |}{t^{2\ell-\alpha-d+1}} \dx t
	\\
	& \leq \sum_{n=0}^\alpha \binom{\alpha}{n}2^{\alpha-n}\bigg(\frac{2P}{\kappa}\bigg)^{\alpha-n} \int_{\kappa/2}^{\infty}  \frac{ | \rho^{(n)}(t)  |}{t^{2\ell-\alpha-d+1}} \dx t
	\\
	& \leq 3^{2\ell}\max_{0\leq n\leq \alpha} \bigg(\frac{2P}{\kappa}\bigg)^{\alpha-n} \int_{\kappa/2}^{\infty}  \frac{ | \rho^{(n)}(t)  |}{t^{2\ell-\alpha-d+1}} \dx t
\end{split}
\end{equation}
which with \eqref{eq:Al-lmm} leads to
\begin{equation}\label{k-03}
\begin{split} 
   | A_\ell  (r)  | & \leq  C_0 6^{2\ell}  (\ell!)^2  ( 2\nu)^{p-\ell}
  \max_{0<\alpha\leq 2\ell \atop 0\leq n\leq \alpha} \bigg(\frac{2P}{\kappa}\bigg)^{\alpha-n} \int_{\kappa/2}^{\infty}  \frac{ | \rho^{(n)}(t)  |}{t^{2\ell-\alpha-d+1}} \dx t\,.
\end{split}
\end{equation}

From Lemma \ref{lem-diff}, we obtain
\begin{equation} \label{k-04}
\begin{split}
	\big|\rho^{(n)} (t)\big|&=\bigg|\frac{2^{1-\nu}}{\Gamma(\nu)} (2\nu)^{\frac{n}{2}}  \sum_{j=0}^{\lfloor {n}/{2}\rfloor } a_{n,j}\big({\sqrt{2\nu}t}\big)^{\nu-j}K_{\nu-n+j}\big({\sqrt{2\nu}t}\big)\bigg|
\end{split}
\end{equation} 
with $a_{n,j}$ as in \eqref{lem-2-eq}, and consequently
\[
\begin{split}
	\int_{\kappa/2}^{\infty}  \frac{ | \rho^{(n)}(t)  |}{t^{2\ell-\alpha-d+1}} \dx t &  \leq  \frac{2^{1-\nu}}{\Gamma(\nu)} (2\nu)^{\frac{n}{2}}\sum_{j=0}^{\lfloor {n}/{2}\rfloor } |a_{n,j}| \int_{\kappa/2}^{\infty}  \frac{\big|{\big({\sqrt{2\nu}t}\big)^{\nu-j}K_{\nu-n+j}\big({\sqrt{2\nu}t}\big)} \big|}{t^{2\ell-\alpha-d+1}} \dx t \\
	&  = \frac{2^{1-\nu}}{\Gamma(\nu)}(2\nu)^{\frac{2\ell-\alpha-d +n}{2}}\sum_{j=0}^{\lfloor {n}/{2}\rfloor } |a_{n,j}| 
	 \int_{t\geq \frac{\kappa \sqrt{2\nu}}{2}}  \frac{K_{\nu-n+j}(t)}{t^{-\nu+j+2\ell-\alpha-d+1}} \dx t\,.
\end{split}
\]
Since $\max\{2\nu -n+j,n-j\}\leq 2p$ we get $|\nu-n+j|\leq 2p-\nu$ which implies $K_{\nu-n+j}(t) \leq K_{2p-\nu}(t) $ by the representation \eqref{Knuintegral}. Again with the assumption $\tau:=\frac{\kappa \sqrt{2\nu}}{2}= \frac{\kappa \sqrt{2\nu}}{2\lambda} \geq 2 P$ and $j+2\ell-\alpha\geq 0$ we can estimate
\[
\begin{split}
	\int_{\kappa/2}^{\infty}  \frac{ | \rho^{(n)}(t)  |}{t^{2\ell-\alpha-d+1}} \dx t & \leq     \frac{2^{1-\nu}}{\Gamma(\nu)} (2\nu)^{\frac{2\ell-\alpha-d+n}{2}}  \int_{t\geq \tau}  \frac{K_{2p-\nu}(t)}{t^{-\nu-d+1}} \dx t \sum_{j=0}^{\lfloor {n}/{2}\rfloor } |a_{n,j}| \\
	& \leq  (2\ell )! \frac{2^{1-\nu}}{\Gamma(\nu)} (2\nu)^{\frac{2\ell-\alpha-d+n}{2}}  \int_{t\geq \tau}  \frac{K_{2p-\nu}(t)}{t^{-\nu-d+1}} \dx t\,,
\end{split}
\]
where in the second step we have used Lemma \ref{lem-diff}. 
Inserting this into \eqref{k-03} we obtain
\[
\begin{split} 
 | A_\ell (r) | & \leq  C_0 6^{2\ell}  (\ell!)^2 (2\ell )! \frac{2^{1-\nu}}{\Gamma(\nu)} \int_{t\geq \tau}  \frac{K_{2p-\nu}(t)}{t^{-\nu-d+1}} \dx t \max_{0<\alpha\leq 2\ell \atop 0\leq n\leq \alpha} \bigg(\frac{2P}{\kappa}\bigg)^{\alpha-n}  (2\nu)^{\frac{2p-\alpha-d+n}{2}}  \\
& \leq C_0 6^{2\ell} (\ell!)^2 (2p )!\frac{2^{1-\nu}}{\Gamma(\nu)} (2\nu)^{p-\frac{d}{2}}\int_{t\geq \tau}  \frac{K_{2p-\nu}(t)}{t^{-\nu-d+1}} \dx t\ \,.
\end{split}
\]
These estimates hold form all $\ell\geq 1$. Moreover, using   \eqref{eq:J} and $|\theta(t)|\leq 1$ for all $t\geq 0$ we have from \eqref{eq:al}
\begin{equation}\label{H0est1}
\begin{split}
|A_0 (r) | 
&
= 
 \big( 2\nu\big)^{p} \bigg| \int_{{\kappa}/{2}}^\infty [\rho\theta](t)(r t)^{-\frac{d}{2}+1}J_{\frac{d}{2}-1}(rt)t^{d-1}\dx t \bigg|
 \\
 & \leq C_0\frac{2^{1-\nu}}{\Gamma(\nu)} ( 2\nu)^{p}  \int_{{\kappa}/{2}}^\infty  \big(\sqrt{2\nu}t\big)^{\nu}K_\nu\big(\sqrt{2\nu}t\big) t^{d-1}\dx t 
\\
& =  C_0\frac{2^{1-\nu}}{\Gamma(\nu)} ( 2\nu)^{p-\frac{d}{2}}  \int_{\tau}^\infty  t^{\nu+d-1}K_\nu(t)\dx t
\\
&\leq  C_0\frac{2^{1-\nu}}{\Gamma(\nu)} ( 2\nu)^{p-\frac{d}{2}}   \int_{\tau}^\infty  t^{\nu +d-1}K_{2p-\nu}(t)\dx t\,.
\end{split}
\end{equation}
Consequently 
\begin{equation} \label{H1est}
|A(r) | \leq C_0 (2\pi)^{d/2} 37^{p} (p!)^2 (2p )!\frac{2^{1-\nu}}{\Gamma(\nu)} (2\nu)^{p-\frac{d}{2}}\int_{t\geq \tau}  \frac{K_{2p-\nu}(t)}{t^{-\nu-d+1}} \dx t\,.
\end{equation}

Since $\nu+d-1 \leq \nu(1+2(d-1))$, as a consequence of  \eqref{besselestimate} with $d_1=1+2(d-1)$,
\[
\begin{split} 
\int_{t\geq \tau}  \frac{K_{2p-\nu}(t)}{t^{-\nu-d+1}} \dx t 
&
\leq 
	\int_{\tau}^\infty t^{\nu d_1} {K_{2p-\nu}(t)} \dx t 
	 \leq e \frac{2^{(4p-2\nu)}\Gamma(2p-\nu)}{2}\int_{\tau}^\infty \frac{ t^{ \nu d_1}e^{-t}}{\sqrt{2t}}\dx t \,.
\end{split}
\]
Using the assumption $\tau = \frac{\kappa \sqrt{2\nu}}{2} \geq  2\nu d_1$ as well as
\[
 	t^{\nu d_1} e^{-t} \leq \biggl(\frac{2 \nu d_1}{e}\biggr)^{\nu d_1} e^{-\frac{t}2},\qquad t>0,
\]
we obtain
\begin{equation} \label{k-008}
	\int_{\tau}^\infty t^{\nu d_1} {K_{2p-\nu}(t)}\, \dx t 
 \leq e \frac{2^{(4p-2\nu)}\Gamma(2p-\nu)}{2\sqrt{\nu}}\bigg(\frac{2\nu d_1}{e} \bigg)^{\nu d_1}  e^{-\frac{\kappa}4 \sqrt{2\nu}} \,.
\end{equation}
Combining \eqref{H1est} and \eqref{k-008}, we arrive at 
\begin{equation*}
	|A(r)  |
	 \leq C_0 (2\pi)^{d/2} 37^{p} (p!)^2 (2p )! (2\nu)^{p-\frac{d}{2}}\frac{2^{(4p-\nu-2)}\Gamma(2p-\nu)}{\Gamma(\nu)\sqrt{\nu}}\bigg(\frac{2\nu d_1}{e} \bigg)^{\nu d_1}   e^{-\frac{\kappa}4 \sqrt{2\nu}}  \,.
\end{equation*}
Now the required bound \eqref{condlambda1} follows from
\begin{equation}\label{eq:sufficient}
  C_{1,\nu} \geq  C_0 (2\pi)^{d/2} 37^{p} (p!)^2 (2p )! (2\nu)^{p-\frac{d}{2}}\frac{2^{(4p-\nu-2)}\Gamma(2p-\nu)}{\Gamma(\nu)\sqrt{\nu}}\bigg(\frac{2\nu d_1}{e} \bigg)^{\nu d_1}   e^{-\frac{\kappa}4 \sqrt{2\nu}} \,.
\end{equation}
Since $(2\nu)^{p-\nu-\frac{d}2} \leq 2\nu $, a sufficient condition for \eqref{eq:sufficient} is
\begin{equation}\label{eq:sufficient2}
 C \geq  37^{p} (p!)^2 (2p )! \frac{2^{(4p-\nu)}\Gamma(2p-\nu)}{ \Gamma(\nu+d/2)  }\bigg(\frac{2\nu d_1}{e} \bigg)^{\nu d_1} \sqrt{\nu}  e^{-\frac{\kappa}4 \sqrt{2\nu}},
\end{equation}
with $C>0$ independent of $\kappa$ and $\nu$. 

Taking logarithms and using the Stirling bounds
\[
\begin{split}
  \ln (p!) &\leq (p +\textstyle\frac12\displaystyle) \ln p  - p + 1, \\
   \ln \Gamma(\nu+d/2) &\geq (\nu + d/2 - \textstyle\frac12\displaystyle ) \ln (\nu+d/2) - (\nu+d/2) + \textstyle\frac12\displaystyle\ln 2\pi ,\\
  \ln \Gamma(2p - \nu) &\leq (2p - \nu - \textstyle\frac12\displaystyle) \ln (2p - \nu) - (2p - \nu) + \textstyle\frac12\displaystyle \ln 2\pi + \textstyle\frac{1}{18}\displaystyle
\end{split}
\]
as well as $p \leq \nu + \frac{d}2 + 1$, for general $\lambda > 0$ shows that the condition
\[
  \frac{\kappa}\lambda \sqrt{\nu} \geq C_1 + C_2 \nu \ln \nu
\]
where $C_1, C_2$ depend only on $d$ (or more precisely, $C_1, C_2 = \mathcal{O}(d \log d)$) is sufficient to ensure \eqref{eq:sufficient2} 
\medskip

\noindent\emph{Step 2. The case $\nu < \frac12$.}
 Since we restrict ourselves to $d\leq 3$, we have $p = \ceil{\nu+\frac{d}2} \leq 2$. Repeatedly using the identities $	K_\nu'(t)=- K_{\nu-1}(t) - \frac{\nu}t K_\nu(t)$ and $K_{\nu}' = \frac{\nu}t K_\nu - K_{\nu+1}$, we obtain
\[
\begin{aligned}
\frac{\dx}{\dx t}\big(t^\nu K_\nu(t)\big) &= -t^\nu K_{\nu-1}(t), 
\\
 \frac{\dx^2}{\dx t^2}\big(t^\nu K_\nu(t)\big) &= t^\nu K_\nu(t) - (2\nu-1) t^{\nu-1}K_{1-\nu}(t), 
 \\
 \frac{\dx^3}{\dx t^3}\big(t^\nu K_\nu(t)\big) &= -t^\nu K_{\nu-1}(t) + (2\nu-1)t^{\nu-1} K_\nu(t) - (2\nu-1)(2\nu-2)t^{\nu-2}K_{\nu-1}(t) 
 \\
 \frac{\dx^4}{\dx t^4} \big(t^\nu K_\nu(t)\big) &= \big[t^\nu + (2\nu-1)(2\nu-2) t^{\nu-2}\big] K_\nu(t) 
 \\ &\quad + \big[ -2(2\nu-1)t^{\nu-1}  -  (2\nu-1) (2\nu-2)(2\nu-3)  t^{\nu-3} \big] K_{\nu-1}(t). 
\end{aligned}
\]
Next, we note that $K_{\nu-1} = K_{1-\nu}$, and by \cite[10.2.17]{AS}
\[ 
  \max\big\{ K_\nu(t), K_{1-\nu}(t) \big\}\leq K_{3/2}(t) = \sqrt{\pi/2} \big(t^{-3/2} + t^{-1/2}\big) e^{-t},
  \]
 where the monotonicity in $\nu$ can be seen from the explicit representation \eqref{Knuintegral}. Continuing from \eqref{k-03}, we now estimate for $0\leq \ell\leq p$
\begin{equation} \label{Aell}
\begin{split} 
 | A_\ell (r)  | & \leq  C ( 2\nu)^{p-\ell}\max_{0<\alpha\leq 2\ell \atop 0\leq n\leq \alpha} \bigg(\frac{2P}{\kappa}\bigg)^{\alpha-n} \int_{\kappa/2}^{\infty}  \frac{ | \rho^{(n)}(t)  |}{t^{2\ell-\alpha-d+1}} \dx t\,.
\end{split}
\end{equation} 
By \eqref{kappacondition} we may use the assumption $\kappa \sqrt{2\nu}/2 \geq  1$, and thus
\[
| \rho^{(n)}(t)  |\leq C\frac{2^{1-\nu}}{\Gamma(\nu)}  (2\nu)^{\frac{n}{2}}e^{-\sqrt{2\nu}t}\,,
\]
if $t\geq \kappa/2$. As a consequence, 
\[
\begin{aligned}
  \int_{\kappa/2}^{\infty}  \frac{ | \rho^{(n)}(t)  |}{t^{2\ell-\alpha-d+1}} \dx t
  & 
  \leq
   C \frac{2^{1-\nu}}{\Gamma(\nu)}  (2\nu)^{\frac{n}{2}} \int_{\kappa/2}^\infty \frac{  e^{-\sqrt{2\nu}t}}{t^{2\ell-\alpha-d+1}} \dx t 
    = 
    C \frac{2^{1-\nu}}{\Gamma(\nu)}  (2\nu)^{\frac{2\ell-\alpha + n -d}{2}}   \int_{\frac{\kappa\sqrt{2\nu}}{2}}^\infty \frac{ e^{-t}}{t^{2\ell-\alpha-d+1}} \dx t 
    \\ 
    & \leq C \frac{2^{1-\nu}}{\Gamma(\nu)}  (2\nu)^{\frac{2\ell-\alpha + n -d}{2}}   e^{-\frac{\kappa\sqrt{2\nu}}{4}}
\end{aligned}
\]
and
\[
\begin{split}
  |A_\ell (r) | 
  &
   \leq 
   C \frac{2^{1-\nu}}{\Gamma(\nu)}   e^{-\frac{\kappa\sqrt{2\nu}}{4}} \max_{0<\alpha\leq 2\ell \atop 0\leq n\leq \alpha} \bigg(\frac{2P}{\kappa}\bigg)^{\alpha-n} (2\nu)^{\frac{2p-\alpha + n -d}{2}} 
   \\
   &
      = 
   C \frac{2^{1-\nu}}{\Gamma(\nu)}   e^{-\frac{\kappa\sqrt{2\nu}}{4}} (2\nu)^{\frac{2p -d}{2}} \max_{0<\alpha\leq 2\ell \atop 0\leq n\leq \alpha} \bigg(\frac{2P}{\kappa\sqrt{2\nu}}\bigg)^{\alpha-n} 
  \\ 
  &
   \leq 
   C \frac{2^{1-\nu}}{\Gamma(\nu)} (2\nu)^{p-\frac{d}{2}}  e^{-\frac{\kappa\sqrt{2\nu}}{4}}\,.
\end{split}
\]
For $0<\nu<\frac12$, the condition for \eqref{eq:eq} to hold becomes 
\[ 
 C_{1,\nu} \left( 2 \nu + \abs{\bsomega}^2\right)^{-\nu-\frac{d}2} \geq C   \left( 2 \nu + \abs{\bsomega}^2 \right)^{-p} \frac{2^{1-\nu}}{\Gamma(\nu)} (2\nu)^{p-\frac{d}{2}}  e^{-\frac{\kappa\sqrt{2\nu}}{4}} ,\qquad \bsomega \in \R^d,
\]
or equivalently
$
e^{\frac{\kappa\sqrt{2\nu}}{4}} \geq C \left( 1 + \abs{\bsomega}^2/(2\nu) \right)^{\nu+\frac{d}{2}-p}
$
with $C>0$ independent of $\kappa$, $\nu$, which is implied by $\kappa \geq (2 \sqrt{2} \ln C)\, \nu^{-1/2}$. This completes the proof.
\end{proof}

\section{Numerical Experiments}\label{sec:num}
The eigenvalue decay established in Theorem \ref{thm:conjecture} has already been studied numerically in \cite{GKNSS}. Note that the results given there are also consistent with the presence of the extra logarithmic factor in \eqref{eq:conjecture}. Here, we thus focus on a numerical study of the required extension size $\gamma$.
In order to assess the sharpness of the necessary conditions \eqref{eq:alphahsmall} and \eqref{kappacondition}, we use a simple bisection scheme to find the minimum value of $\gamma$ that is actually required in each case to ensure that the obtained covariance matrix is positive definite. In all tests, we assume the box $[-1,1]^d$ as the computational domain, and we show results only for $\lambda=\frac12$ since the resulting values of $\gamma$ exhibit an approximately linear scaling with respect to $\lambda$.

In the case of smooth periodization, in addition to the cutoff function using integrated B-splines defined in \eqref{eq:bspline1}, \eqref{eq:bsplined} we also test a standard infinitely differentiable cutoff function as used in \cite{BCM}, which is simpler to implement in practice: let  
\[  
 \eta(x) = \begin{cases}
 	  \exp(-x^{-1}), & x > 0,\\
 	   0, & x \leq 0.
 \end{cases}
 \]
One can then replace the definition of $\varphi$ in \eqref{eq:bspline1} by 
\begin{equation}\label{eq:smooth1}
  \varphi(t) = \frac{\eta\left(\frac{\kappa-\abs{t}}{\kappa-1}\right) }{ \eta\left(\frac{\kappa-\abs{t}}{\kappa-1}\right) + \eta\left(\frac{\abs{t}- 1}{\kappa-1}\right)},
  \end{equation}
  and again define $\varphi_\kappa(\bx) := \varphi(\abs{\bx})$.

First, we compare the extension sizes $\gamma$ in terms of the grid size $h$ that are needed for classical circulant embedding and for smooth periodization; Figure \ref{fig:ratio} shows the resulting ratio of the number of grid points in the extension torus $\mathbb{T} = [-\gamma,\gamma]^d$ to the number of sampling grid points in the original domain. Whereas, as expected in view of Theorem \ref{thm:smoothcond}, the minimum required values of $\gamma$ are indeed independent of $h$ in the case of the smooth truncation, in the case of the classical circulant embedding the corresponding values of $\gamma$ indeed exhibit a dependence of order $\abs{\log h}^d$ on $h$. In this sense, we observe the result of Theorem \ref{thm:matern-growth} to be sharp. Especially for $d=3$ and smaller values of $h$, the smooth periodization leads to substantially more favorable extension sizes. The results shown are for the $C^\infty$-cutoff function \eqref{eq:smooth1}, and one obtains very similar results with \eqref{eq:bspline1}.
\begin{figure}
\includegraphics[width=7.2cm]{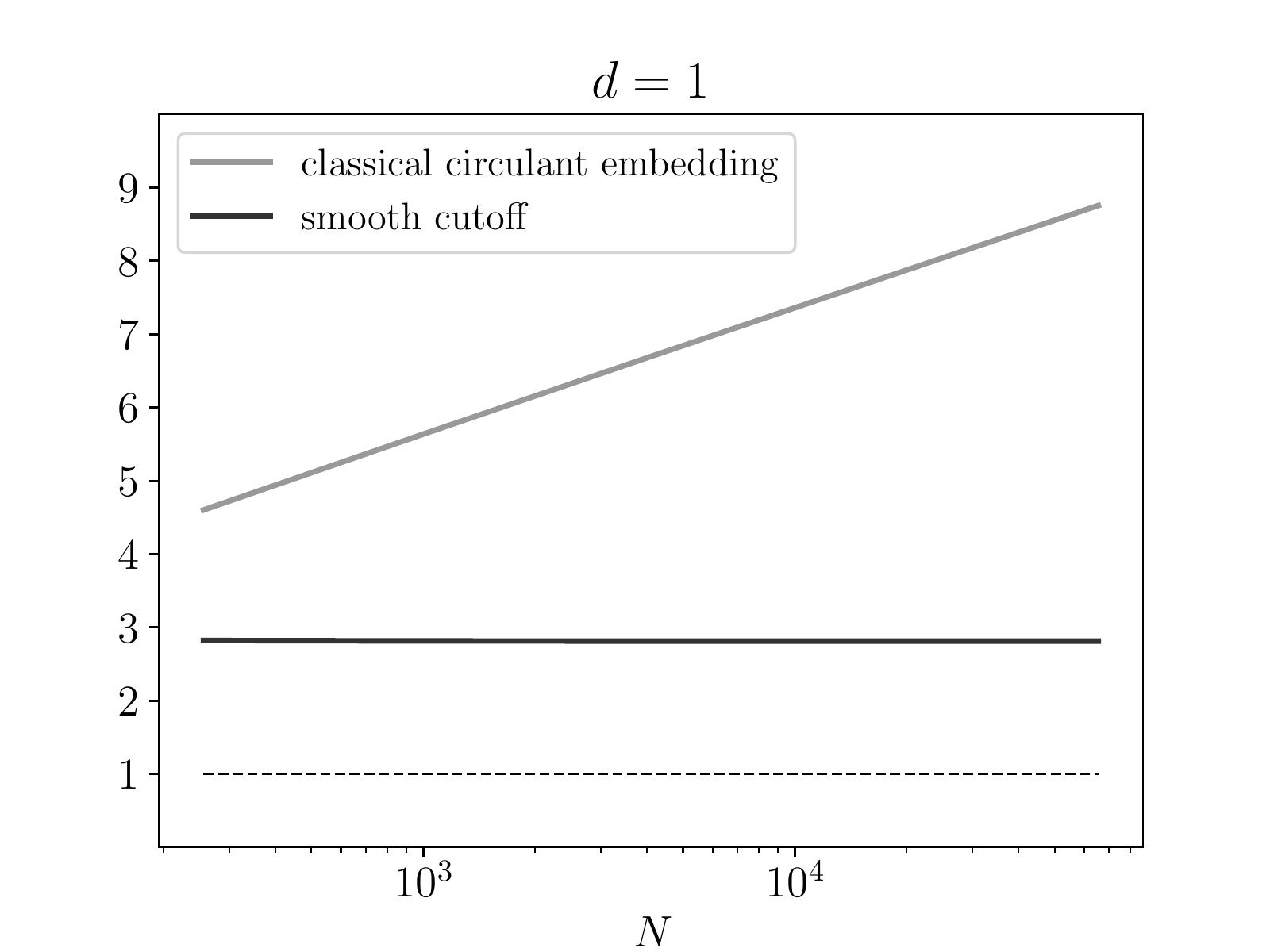} \\
\includegraphics[width=7.2cm]{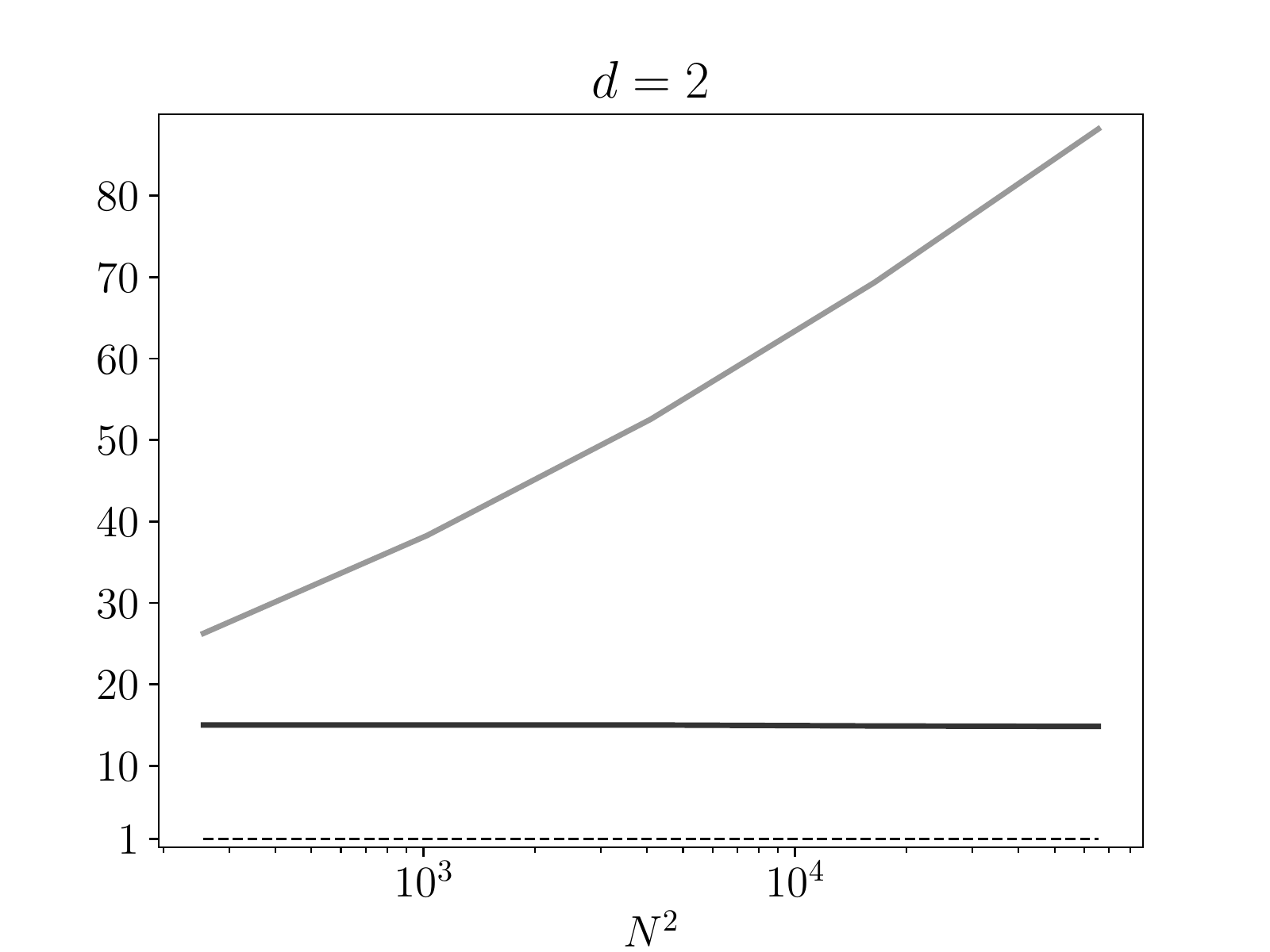} 
\includegraphics[width=7.2cm]{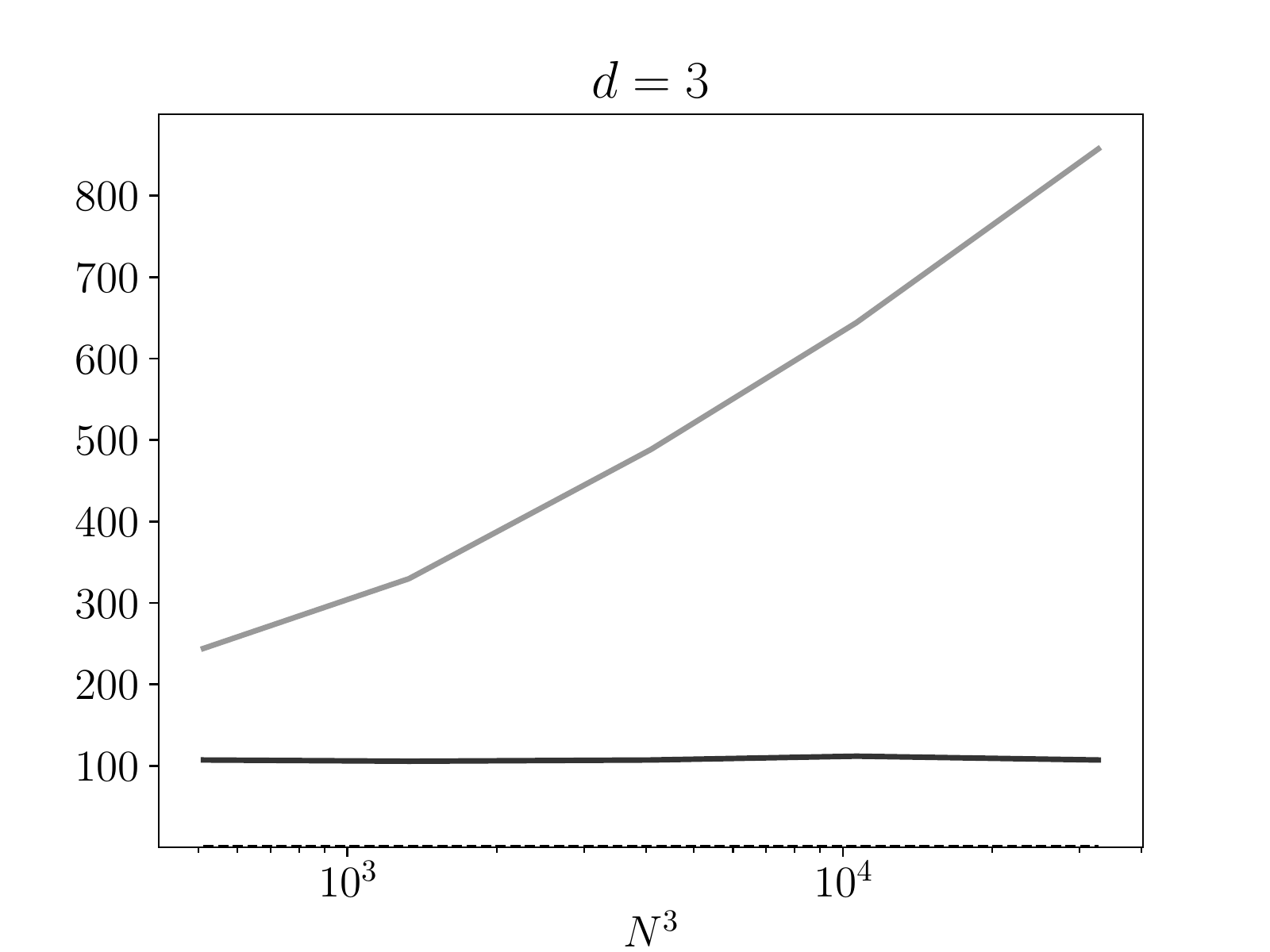}
\caption{Comparison of the number of grid points in the extension to the original number of grid points for classical circulant embedding and for smooth periodization, and for $\nu = 1$ and $h = 2^{-8}, \ldots, 2^{-16}$ in $d=1$; $h=2^{-4}, \ldots, 2^{-8}$ in $d=2$; and $h=2^{-3}, \ldots, 2^{-5}$ in $d=3$. \rev{The same legend applies to all $d$.}}
\label{fig:ratio}
\end{figure}

In addition, in Figure \ref{fig:nu} we consider the dependence of the minimum required $\gamma$ on $\nu$ for the periodization with smooth truncation and compare to the asymptotics in the condition \eqref{kappacondition} of Theorem \ref{thm:smoothcond}. 
We show results for both cutoff function constructions \eqref{eq:bspline1} and \eqref{eq:smooth1}.
We observe that in the particular case $d=1$, $\gamma$ remains bounded as $\nu \to 0$ (indeed, we observe $\gamma\to 1$ in this limit), whereas for $d>1$ we find an increase in $\gamma$ both as $\nu\to 0$ and $\nu \to \infty$. The actual required increase of $\gamma$ as $\nu \to 0$ appears to be slightly slower than the order $\nu^{-1/2}$ in \eqref{kappacondition}. The observed behaviour of $\gamma$ for larger $\nu$ is consistent with the sufficient condition of order $\nu^{1/2}\log \nu$ in \eqref{kappacondition}. Note that the B-spline cutoff of limited smoothness leads to a slower increase of $\gamma$ as $\nu\to \infty$, whereas $\gamma$ in this case increases slightly faster as $\nu\to 0$ for $d>1$.
\begin{figure}
\includegraphics[width=7.2cm]{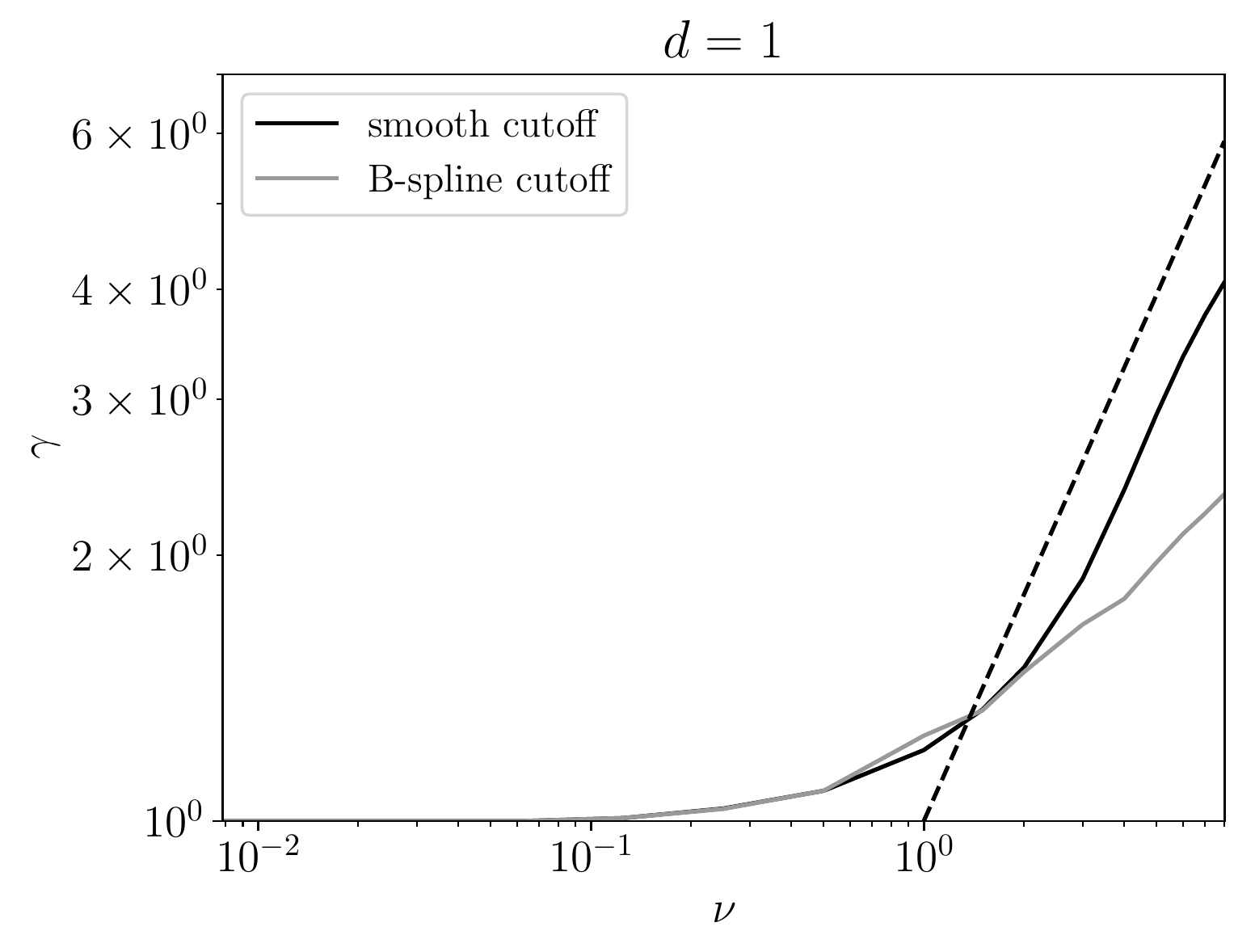} \\
\includegraphics[width=7.2cm]{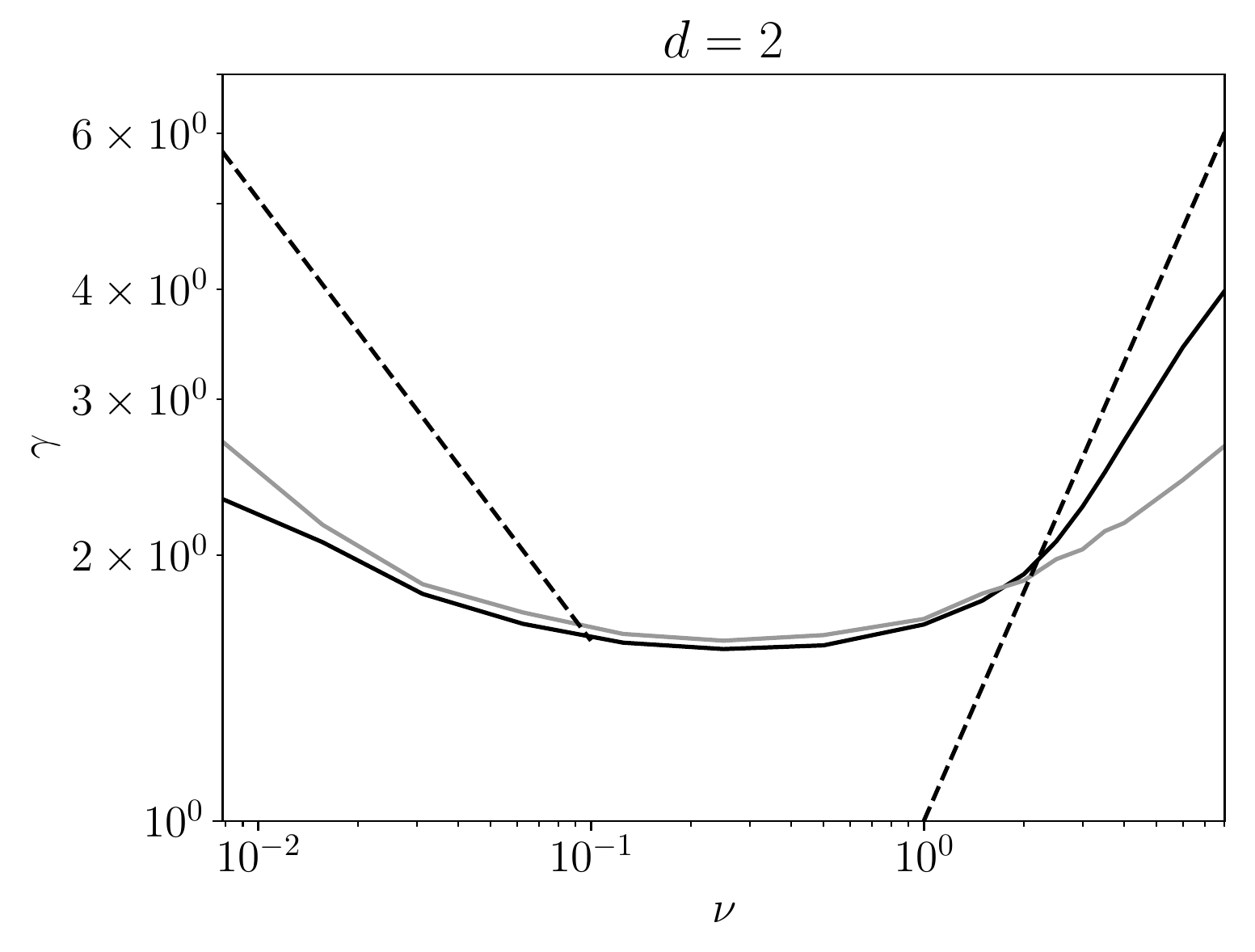} 
\includegraphics[width=7.2cm]{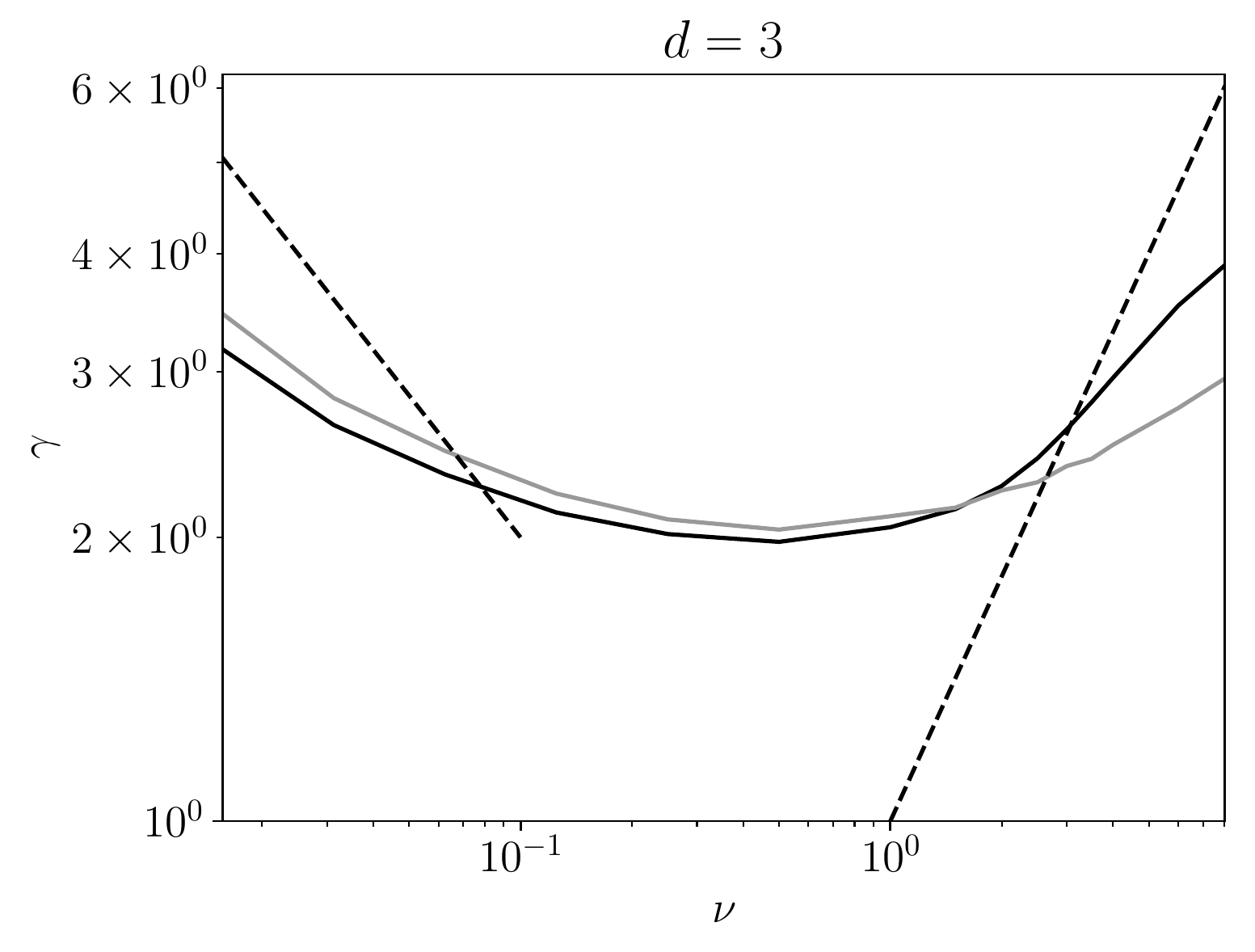}
\caption{Minimum extension size $\gamma$ required for positive
  definiteness of the covariance matrix resulting from smooth
  periodization in dependence on $\nu$, with $\nu = 2^{-7}, 2^{-6},
  \ldots, 2^2, 2^3$ and for $d=1,2,3$ with $h = 1/40000, 1/800, 1/30$,
  respectively. Dashed lines show the asymptotics $\nu^{-1/2}$ and
  $\nu^{1/2}\log \nu$ from \eqref{kappacondition}. 
\rev{The same legend applies to all $d$.}}
\label{fig:nu}
\end{figure}

\section{\rev{Conclusions}}\label{sec:conclusions}

\rev{We have seen that both classical and smooth periodization preserve the asymptotic decay rate of covariance eigenvalues.
Concerning the factor $\gamma \geq 1$ by which the sampling grid needs to be extended to ensure positive definiteness, there is a \rev{difference:
the smooth periodization requires an extension factor $\gamma$ that is
independent of the sample grid size $h$, whereas classical
periodization requires} $\gamma \sim \abs{\log h}$ according to
Theorem \ref{thm:matern-growth}. For the extended grid size $N^d = (2\gamma h^{-1})^d$ we thus have $N^d \sim h^{-d}$ in the case of smooth periodization, and $N^d \sim h^{-d} \abs{\log h}^d$ in the case of classical periodization. 
This is exactly what we observe for the numerically determined minimum extension sizes in Figure \ref{fig:nu}.
\rev{This directly translates to} efficiency advantages of the smooth periodization that are more pronounced for larger $d$.
The total \rev{cost is} dominated by the FFT on the \rev{extended} grid, which requires $\mathcal{O}(N^d \log N)$ operations.
Whereas in the case of the classical circulant embedding, one sample on $D$ thus
costs $\mathcal{O}(h^{-d} \abs{\log h}^{d+1})$ operations,
with the smooth periodization the cost is only $\mathcal{O}(h^{-d}
\abs{\log h})$ \rev{in terms of $h$}.}

\rev{One further useful consequence of the new result \eqref{eq:condnew} covering the full range of $\lambda$, $\nu$ is that smooth periodization also provides an attractive way of sampling with \emph{hyperpriors} on these two parameters: in this case, one needs to first randomly select $\lambda$ and $\nu$ according to some probability distribution, and then draw a sample of the Gaussian random field with the corresponding Mat\'ern covariance. This problem has been addressed, for instance, in \cite{LEU} by approximate sampling using a reduced basis. 
In this context, it is relevant that for both types of periodization, setting up the factorization for a new covariance and drawing a sample come at the same cost of one FFT each: thus by \eqref{eq:condnew}, provided that the distributions of $\lambda$, $\nu$ have compact supports inside $(0,\infty)$, smooth  periodization is guaranteed to provide each sample -- without any further approximation --  at near-optimal cost $\mathcal{O}(h^{-d} \abs{\log h})$.}

\rev{The periodic random field $Z^\mathrm{p}$ on $\mathbb{T}$, obtained
  using smooth truncation, also provides a tool for deriving
  \emph{series expansions} of the original \rev{random field. As
    mentioned above, in the smooth periodization case} the KL eigenvalues of the
periodic field have the decay rate \eqref{kldecay}. Moreover, in contrast to the KL  eigenfunctions on
$D$, which are typically not explicitly known, the corresponding
eigenfunctions $\varphi^\mathrm{p}_j$ of the periodic covariance are
explicitly known trigonometric functions \rev{and one has the following KL
expansion for the periodized random field:}
\[
  Z^\mathrm{p} = \sum_{j=1}^\infty y_j \sqrt{\lambda^\mathrm{p}_j} \,\varphi^\mathrm{p}_j,
  \quad y_j \sim \mathcal{N}(0,1)\  \text{ i.i.d.,} 
\]
with \rev{$\lambda^\mathrm{p}_j$ denoting} the eigenvalues of the
periodized covariance \rev{and the $\varphi^\mathrm{p}_j$ are normalized in
  $L^2(\mathbb{T})$}.
Restricting this expansion back to $D$,  one obtains an \rev{exact} expansion of the original random field on $D$ in terms of independent scalar random variables
\begin{equation}\label{nondstdexpansion}
  Z  = \sum_{j=1}^\infty y_j \sqrt{\lambda^\mathrm{p}_j} \bigl(\varphi^\mathrm{p}_j\big|_D\bigr), \quad y_j \sim \mathcal{N}(0,1)\  \text{ i.i.d.} 
\end{equation}
This provides an alternative to the standard KL expansion
\eqref{standardkl} of $Z$ in terms of eigenvalues $\lambda_j$ and
eigenfunctions $\varphi_j$ \rev{normalized in $L^2(D)$. The main
  difference is that the functions $\varphi^\mathrm{p}_j\big|_D$ in
\eqref{nondstdexpansion}}
are not $L^2(D)$-orthogonal. However, these functions are given
explicitly, and thus no approximate computation of eigenfunctions is
required.}

\rev{The eigenvalues $\lambda^\mathrm{p}_j$ can be approximated
  efficiently by FFT as described in \cite[Sec.\ 5.1]{BCM} \rev{and the
  asymptotic decay of 
 $\lambda^\mathrm{p}_j$ is the same as that of
 $\lambda_j$ in \eqref{standardkl}. Moreover, since the eigenfunctions of the periodized covariance
  are explicitly known trigonometric functions, the
 $L^\infty$-norms of the scaled eigenfunctions in
 the expansion \eqref{nondstdexpansion} also decay at
the same rate as 
$\sqrt{\lambda_j}$, that is,
\[
\left\| \sqrt{\lambda^\mathrm{p}_j}\varphi^\mathrm{p}_j|_D \right\|_{L^\infty(D)} \leq C
  \sqrt{\lambda_j}.
\] 
This decay
  of the $L^\infty$-norms is important for applications, e.g., to random
  PDEs. Since $\norm{\varphi_j}_{L^\infty(D)}$ is in general
  not uniformly bounded, see \cite[Sec.\ 3]{BCM}, this constitutes a
  marked advantage of the expansion \eqref{nondstdexpansion} over the
  standard KL expansion \eqref{standardkl}. In addition, on
  geometrically complicated $D$, 
\eqref{nondstdexpansion} is significantly easier to handle numerically.}}

\rev{The KL expansion of $Z^p$ also enables the construction of alternative expansions of $Z$ of the basic form \eqref{nondstdexpansion}, but with the spatial functions having additional properties.
In \cite{BCM}, wavelet-type representations 
\[
  Z = \sum_{\ell,k} y_{\ell,k} \psi_{\ell, k} , 
  \quad y_{\ell,k} \sim \mathcal{N}(0,1)\ \text{ i.i.d.,}
\]
are constructed \rev{by applying the square root of the covariance
  operator in the corresponding factorisation to periodic Meyer
  wavelets, with the summation running over $\ell \geq 0$ and $ k \in
\{0,\ldots,2^\ell -1\}^d$. The functions $\psi_{\ell,k}$} have the same multilevel-type localisation as the Meyer wavelets. This feature yields improved convergence estimates for tensor Hermite polynomial approximations of solutions of random diffusion equations with lognormal coefficients \cite{BCDM}.}

\bibliographystyle{amsplain}
\bibliography{BGNS_ceanalysis}

\begin{appendix}
	\section{Proofs of Auxiliary Results}\label{auxproofs}

	\begin{proof}[Proof of Lemma \ref{lem-estimate}.] If $0\leq \nu\leq \frac{1}{2}$ we use the estimate $K_\nu(t)\leq K_{1/2}(t)=\sqrt{\frac{\pi}{2t}} e^{-t} $. When $\nu>1/2$ we have (see \cite[Page 206]{Wat})
		\[ 
		\begin{split} 
		K_\nu(t) & =\Big(\frac{\pi}{2t}\Big)^{1/2} \frac{e^{-t}}{\Gamma(\nu+1/2)} \int_0^\infty e^{-u}u^{\nu-\frac{1}{2}}\Big(1+\frac{u}{2t}\Big)^{\nu-\frac{1}{2}}\dx u \\
		& \leq \Big(\frac{\pi}{2t}\Big)^{1/2} \frac{e^{-t}}{\Gamma(\nu+1/2)} \int_0^\infty e^{-u}u^{\nu-\frac{1}{2}}\big(1+u\big)^{\nu-\frac{1}{2}}\dx u \qquad (t\geq 1/2)\\
		& \leq \Big(\frac{\pi}{2t}\Big)^{1/2} \frac{e^{-t}}{\Gamma(\nu+1/2)} \int_0^\infty e^{-u} \big(1+u\big)^{2\nu-1}\dx u\,.
		\end{split}
		\]
		Changing variable we obtain
		\[
		\begin{split}
		K_\nu(t) 
		&\leq e \Big(\frac{\pi}{2t}\Big)^{1/2} \frac{e^{-t}}{\Gamma(\nu+1/2)} \int_1^\infty e^{-u} u^{2\nu-1}\dx u \leq  e \Big(\frac{\pi}{2t}\Big)^{1/2} \frac{e^{-t}\Gamma(2\nu)}{\Gamma(\nu+1/2)}\,.
		\end{split}
		\]
		Now using the duplication formula $\Gamma(\nu)\Gamma(\nu+1/2)=2^{1-2\nu}\sqrt{\pi}\Gamma(2\nu)$ we get the desired result.
	\end{proof}
	\begin{proof}[Proof of Lemma \ref{lem-diff}.]
		We proceed by induction. The statement in Lemma \ref{lem-diff} holds for $n=0$ and $n=1$ since $	\frac{\dx}{\dx t}\big(t^\nu K_\nu(t)\big)=-t^\nu K_{\nu-1}(t)$ which follows from $tK_\nu'(t)+\nu K_\nu(t)=-tK_{\nu-1}(t)$, see \cite[Section 3.71]{Wat}. Assume that  the statement holds for some $n\geq 1$. We write
		\[ 
		\frac{\dx^n }{\dx t^n}\big(t^\nu K_\nu(t)\big) 
		= 
		\sum_{j=0}^{\lfloor {n}/{2}\rfloor } a_{n,j}t^{\nu-j}K_{\nu-n+j}(t)
		=
		\sum_{j=0}^{\lfloor {n}/{2}\rfloor } a_{n,j}t^{\nu-n+j}K_{\nu-n+j}(t) t^{n-2j}\,.
		\]
		Taking derivatives on both sides and using $	\frac{\dx}{\dx t}\big(t^\tau K_\tau(t)\big)=-t^\tau K_{\tau-1}(t)$ we get
		\[ 
		\frac{\dx^{n+1} }{\dx t^{n+1}}\big(t^\nu K_\nu(t)\big) 
		=
		\sum_{j=0}^{\lfloor {n}/{2}\rfloor } a_{n,j}\big[-t^{\nu-j}K_{\nu-n+j-1}(t) + (n-2j) t^{\nu-j-1}K_{\nu-n+j}(t)\big] \,.
		\]
		This is of the form given in \eqref{lem-2-eq}, moreover
		\[
		\sum_{j=0}^{\lfloor {(n+1)}/{2}\rfloor } |a_{n+1,j}|
		\leq
		\sum_{j=0}^{\lfloor {n}/{2}\rfloor } |a_{n,j}| \big(1+ n-2j\big) \leq (n+1) \sum_{j=0}^{\lfloor {n}/{2}\rfloor } |a_{n,j}| \leq (n+1)!,
		\]
		where we have used the induction hypothesis. The proof is completed. 
	\end{proof}
	\begin{proof}[Proof of Lemma \ref{lmm:sigmaest}]
		For $|\bsalpha|_\infty \leq 2$ and $|\bx|\geq \gamma/2$, we have
		\[
		\begin{aligned}
		\partial^{\bsalpha}\sigma(\bx) &=  \partial^{\bsalpha}\big[\rho\big(1-\phi_1(|\cdot|/\gamma)\big) \big](\bx)
		\\
		&=  \partial^{\bsalpha}\rho(\bx)- \sum_{{0}\leq \bsbeta\leq \bsalpha} \binom{\bsalpha}{\bsbeta}\partial^{\bsalpha-\bsbeta} \rho(\bx)\,\partial^{\bsbeta}\big[\phi_1(|\cdot|/\gamma)  \big](\bx)
		\\
		& = \partial^{\bsalpha}\rho(\bx) - \sum_{{0}\leq \bsbeta\leq \bsalpha} \gamma^{-|\bsbeta|}\binom{\bsalpha}{\bsbeta}\partial^{\bsalpha-\bsbeta}\rho(\bx)\, \partial^{\bsbeta}\big[\phi_1(|\cdot|) \big](\bx/\gamma) ,
		\end{aligned}
		\]
		which implies, since $\gamma\geq 1$,
		\[
		| \partial^{\bsalpha}\sigma(\bx)| \leq C \max_{{0}\leq \bsbeta\leq \bsalpha}|\partial^{\bsbeta}\rho(\bx)| 
		= C\frac{2^{1-\nu}}{\Gamma(\nu)}\max_{{0}\leq \bsbeta\leq \bsalpha}\bigg|\partial^{\bsbeta}\bigg[\bigg(\frac{\sqrt{2\nu}|\bx|}{\lambda}\bigg)^{\nu}K_\nu\bigg(\frac{\sqrt{2\nu}|\bx|}{\lambda}\bigg)\bigg]\bigg|
		\]
		for some positive constant $C$ depending on $\bsalpha$. Since $r=|\bx| \geq \gamma/2$ implies in particular $r \geq 1/2$, we have  $|\partial^{\bsalpha} r| \leq C$ (which in general depends only on $d$) and therefore obtain from Lemma \ref{lem-diff}
		\[
		\begin{split} 
		| \partial^{\bsalpha}\sigma(\bx)|
		&
		\leq
		C \max_{{0}\leq \bsbeta\leq \bsalpha} \max_{n\leq |\bsbeta|} 	\bigg|\frac{\dx^n }{\dx r^n}\big[r^\nu K_\nu(r)\big] \bigg(\frac{\sqrt{2\nu}|\bx|}{\lambda}\bigg)\bigg| 
		\\
		&
		\leq 
		C \max_{n\leq 2d} 	\bigg|\sum_{j=0}^{\lfloor {n}/{2}\rfloor } a_{n,j}\bigg(\frac{\sqrt{2\nu}|\bx|}{\lambda}\bigg)^{\nu-j}K_{\nu-n+j}  \bigg(\frac{\sqrt{2\nu}|\bx|}{\lambda}\bigg)\bigg| \,.
		\end{split}
		\]
		Note that, using $K_\nu = K_{-\nu}$, by the monotonicity of $K_\nu$ in $\nu$ for $\nu\geq 0$ and the series representation of $K_{\nu}$ \cite[Page 80]{Wat}, we can estimate, for all $t \geq 0$, 
		\begin{align} \label{series}
		K_{\nu}(t) \leq K_{\ell-1/2}(t)=\sqrt{\frac{\pi}{2t}} e^{-t} \sum_{k=0}^{\ell-1} \frac{(\ell-1+k)!}{k!(\ell-1-k)!(2t)^k}, \qquad   \ell=\ceil{\abs{\nu} +1/2}. 
		\end{align}
	Now, if    $|\bx| \geq \gamma/2$,  then $t := \sqrt{2 \nu} \vert \bx \vert/\lambda\geq C$,  \eqref{series} implies that $K_\nu(t) \leq C t^{-1/2} e^{-t}$ and so 
		\[
		\begin{aligned} 
		| \partial^{\bsalpha}\sigma(\bx)|
		&
		\leq 
		C  	  \max_{j\leq d}\bigg(\frac{\sqrt{2\nu}|\bx|}{\lambda}\bigg)^{\nu-j-\frac{1}{2}}e^{  -\frac{\sqrt{2\nu}|\bx|}{\lambda}}
		\leq C\bigg(\max_{|\bx|\geq \gamma/2}e^{  -\frac{\sqrt{2\nu}|\bx|}{2\lambda}}\bigg) \bigg(\frac{\sqrt{2\nu}|\bx|}{\lambda}\bigg)^{\nu}e^{  -\frac{\sqrt{2\nu}|\bx|}{2\lambda}} \,
		\\
		&
		\leq 
		C e^{{- \frac{\sqrt{2\nu}}{4 \lambda}\gamma}}  \bigg(\frac{\sqrt{2\nu}|\bx|}{\lambda}\bigg)^{\nu}e^{  -\frac{\sqrt{2\nu}|\bx|}{2\lambda}}\,. \qedhere
		\end{aligned}
		\]
	\end{proof}
	\begin{proof}[Proof of Lemma \ref{derivative}] The assertion is obvious for $d=1$. With $\ell=1$ we have
		$
		B_{2,d}= f''(t)t^{d-1} + (d-1)f'(t)t^{d-2}\,.
		$
		Assuming that \eqref{k-001} is true for $\ell$, we prove the this statement for $\ell+1$. We have
		\[
		\begin{split}
		B_{ 2\ell +1,d}(f,t)
		&
		=\sum_{\alpha =1}^{ 2\ell } a_{ 2\ell ,\alpha } \big[f^{(\alpha )}(t)t^{\alpha - 2\ell }\big]'
		\\
		& = \sum_{\alpha =1}^{ 2\ell } a_{ 2\ell ,\alpha }\big[ f^{(\alpha +1)}(t)t^{\alpha - 2\ell }+ (\alpha - 2\ell )f^{(\alpha )}(t)t^{\alpha - 2\ell -1}\big] 
		\end{split}
		\]
		and
		\begin{equation}\label{bm}
			\begin{split}
				B_{ 2\ell +2,d}(f,t) & = \sum_{\alpha =1}^{ 2\ell } a_{ 2\ell ,\alpha }\big[ f^{(\alpha +1)}(t)t^{d-1+\alpha - 2\ell }+ (\alpha - 2\ell )f^{(\alpha )}(t)t^{d+\alpha - 2\ell -2}\big]' \\
				& = \sum_{\alpha =1}^{ 2\ell } a_{ 2\ell ,\alpha }\Big[ f^{(\alpha +2)}(t)t^{d-1+\alpha - 2\ell }+ \big(d-1+2(\alpha - 2\ell )\big)f^{(\alpha +1)}(t)t^{d-2+\alpha - 2\ell }\\
				&\qquad\qquad\qquad\  + (\alpha - 2\ell )(d-2+\alpha - 2\ell )f^{(\alpha )}(t)t^{d-3+\alpha - 2\ell }\Big]\,.
			\end{split}
		\end{equation}	
		When $d=2$ we obtain $\sum_{\alpha=1}^2 \abs{a_{2,\alpha}} = d = 2$ and
		\[ 
		\sum_{j =1}^{ 2\ell +2} |a_{ 2\ell +2,j }| \leq \sum_{\alpha =1}^{ 2\ell } |a_{ 2\ell ,\alpha }|\big[( 2\ell -\alpha )^2 + |1+ 2(\alpha- 2\ell )| +1\big] \leq 4\ell^2 \sum_{\alpha =1}^{ 2\ell } |a_{ 2\ell ,\alpha }|.
		\]
		This proves \eqref{k-001}. When $d=3$, by the induction assumption we have $$B_{ 2\ell ,3}(f,t)= f^{( 2\ell )}(t)t^{2}+ 2\ell f^{( 2\ell -1)}(t)t.$$ 
		Using \eqref{bm} again we obtain
		\[ 
		\begin{split} 
		B_{ 2\ell +2,3}(f,t) &= f^{(2\ell+2)}(t)t^2 + 2f^{(2\ell+1)}(t)t + 2\ell\big[f^{(2\ell+1)}(t)t \big]
		\\
		& = f^{( 2\ell+2 )}(t)t^{2}+ (2\ell+2) f^{( 2\ell +1)}(t)t\,.
		\end{split}
		\]
		This completes the proof of the lemma.
	\end{proof}

\end{appendix}
\end{document}